\documentclass[11pt]{amsart}
\usepackage{amssymb, graphicx}
\usepackage{color}
\usepackage{amsfonts}
\usepackage{amsthm}
\usepackage{enumerate}
\usepackage{comment}
\usepackage[mathscr]{eucal}
\usepackage{graphicx}

\usepackage[pdftex, bookmarksnumbered, bookmarksopen, colorlinks, citecolor=blue, linkcolor=blue]{hyperref}

\allowdisplaybreaks

\newtheorem{thm}{Theorem}[section]

\newtheorem{cor}[thm]{Corollary}
\newtheorem{lem}[thm]{Lemma}

\theoremstyle{definition}
\newtheorem{remark}[thm]{Remark}

\theoremstyle{remark}
\numberwithin{equation}{section}

\newcommand{\R}{{\mathbb R}}
\newcommand{\N}{{\mathbb N}}
\newcommand{\Z}{{\mathbb Z}}

\newcommand{\M}{{\mathcal M}}

\newcommand{\n}{{\mathcal N}}
\newcommand{\bs}{\begin{split}}
\newcommand{\es}{\end{split}}
\newcommand{\F}{\mathcal F}

\theoremstyle{plain}

\theoremstyle{definition}

\theoremstyle{plain}

\theoremstyle{definition}

\theoremstyle{plain}

\newcommand{\be}{\begin{eqnarray*}}
\newcommand{\ee}{\end{eqnarray*}}
\newcommand{\beq}{\begin{align}}
\newcommand{\eeq}{\end{align}}

\def\esssup{\mathop{\mathrm{ess \ sup}}}
\def\Q{\mathcal{Q}}
\def\1{\mathbf{1}}

\begin{document}

%
%
%
%
%
%
%
%
\setcounter{page}{1}
\title[Rough operators]
{A noncommutative weak $L_1$-estimate for rough singular integrals}

\author[G. Hong]{Guixiang Hong}
\address{
	Institute for Advanced Study in Mathematics\\
	Harbin Institute of Technology\\
	Harbin 150001\\
	China}

\email{gxhong@hit.edu.cn}

\author{Liang Wang}
\address{Department of Mathematics\\ City University of Hong Kong\\Hong Kong SAR\\China}
\email{L.Wang@cityu.edu.hk}

\author[X. Wang]{Xumin Wang}
\address{
	Institute for Advanced Study in Mathematics\\
	Harbin Institute of Technology\\
	Harbin 150001\\
	China}

\email{xumin.wang@hit.edu.cn}
\author[B. Xu]{Bang Xu}
\address{School of Mathematical Sciences\\ Xiamen University\\
	Xiamen 361005\\
	China}

\email{bangxu@xmu.edu.cn}



\begin{abstract}
In this paper, we establish a noncommutative version of Andreas Seeger's weak type $(1,1)$ estimate for singular integral operators  with kernels satisfying angular $L\log L$ integrability and an averaged radial Dini condition. The strategy for obtaining the sharp integrability condition involves not only some new estimates of the off-diagonal good and bad functions in Parcet's noncommutative Calder\'on-Zygmund decomposition but also certain essential modifications of Seeger's kernel decomposition and estimates.
\end{abstract}

\maketitle

\section{Introduction}

	The purpose of this paper is to establish weak type $(1,1)$ estimates for
rough convolution operators acting on functions with values in a semifinite
von Neumann algebra.

In the classical setting, an important special case of rough singular integrals occurs if the convolution kernel is homogeneous of degree $-d$ with $d\geq2$. Suppose that $\Omega\in L_1(S^{d-1})$ and
\begin{align}\label{cancellation}
\int_{S^{d-1}}\Omega(\theta)d\sigma(\theta)=0,
\end{align}
where $d\sigma(\theta)$ is the sphere measure of $S^{d-1}$.  The associated rough homogeneous singular integral is initially defined on smooth compactly
supported functions by
\begin{align*}
	T_{\Omega}f(x)=\mbox{p.v.}\int_{\R^d}
	\frac{\Omega({y/|y|})}{|y|^{d}}f(x-y)dy.
\end{align*}
In what follows, $L\log L(S^{d-1})$ denotes the class determined by
\[
\int_{S^{d-1}}|\Omega(\theta)|
\bigl(1+\log_+|\Omega(\theta)|\bigr)\,d\sigma(\theta)<\infty.
\]

The theory originates in the work of Calder\'on and Zygmund
\cite{CZ1952,CZ1956}. Their method of rotations gives strong $L_p$-boundedness for $1<p<\infty$ whenever $\Omega\in L\log L(S^{d-1})$. The endpoint question asks whether
\[
\bigl|\{x\in\R^d:|T_\Omega f(x)|>\lambda\}\bigr|
\leq \frac{C_\Omega}{\lambda}\|f\|_{L_1(\R^d)},
\qquad \lambda>0.
\]
The first breakthrough is due to Christ and Rubio de Francia \cite{CR1988}, who obtained the weak type $(1,1)$ estimate of $T_{\Omega}$ under the $L\log L$ assumption in dimension two; see also \cite{Hofmann1988} for a similar result under a stronger hypothesis.  In higher dimensions, some closely related results for Bochner-Riesz means with critical regularity index and rough Hardy-Littlewood maximal operators were also obtained in  \cite{Christ1988,CR1988}.
Motivated by all this, Seeger \cite{Seeger1996JAMS} finally extended Christ and Rubio de Francia's result to arbitrary dimension by introducing the microlocal decomposition that enabled him to  separate
spatial directions and frequency regions
so that radial oscillation can be used without differentiating the angular
factor; see also Tao's similar result on homogeneous groups in \cite{Tao1999}. 

Actually, Seeger \cite{Seeger1996JAMS}  obtained the endpoint estimate for a more general class of convolution kernels.
Let $K\in\mathcal S'(\R^d)$ be a tempered distribution
represented by a locally integrable function away from the origin.
An operator associated to this kernel is defined by
\begin{align}\label{general case}
\langle Tf,g\rangle
=\int_{\R^d}\int_{\R^d}
K(x-y)f(y)\overline{g(x)}\,dy\,dx
\end{align}
for all $f,g\in C_c^\infty(\R^d)$ with disjoint supports. Clearly, $T$ extends to a bounded operator on $L_2(\R^d)$ if and only if $m=\widehat K\in L_\infty(\R^d)$, where $\widehat K(\xi)=\int e^{-2\pi i x\cdot\xi}K(x)\,dx$. Under the $L_2$-boundedness assumption, a standard sufficient regularity condition for weak type $(1,1)$ boundedness in classical theory is the H\"ormander condition
\begin{align}\label{hormander}
\sup_{y\ne0}\int_{|x|>2|y|}|K(x-y)-K(x)|\,dx<\infty.
\end{align}
But Seeger considered the following conditions, which, roughly speaking, impose regularity in each direction and only integrability in the angular variable.
	Put
	\[
	V(\theta)
=\sup_{R>0}V_R(\theta)=\sup_{R>0}\int_R^{2R}|K(r\theta)|r^{d-1}\,dr,
	\qquad
	W=\int_{S^{d-1}}V(\theta)\,d\sigma(\theta).
	\]
	For \(a\geq2\), define
	\[
	\eta(a)
	=\sup_{\substack{v>0\\R\geq av}}
	\int_{S^{d-1}}\int_R^{2R}
	|K((r-v)\theta)-K(r\theta)|r^{d-1}\,dr\,d\sigma(\theta),
	\qquad
	D=\int_2^\infty\eta(a)\,\frac{da}{a}.
	\]
We use $\log_+ a=\max\{0,\log a\}$ for $a>0$ with $\log_+0=0$. Set
\begin{equation*}
	L=\int_{S^{d-1}}V(\theta)\bigl(1+\log_+ V(\theta)\bigr)\,d\sigma(\theta),
	\qquad \Lambda_K=1+D+W+L.
\end{equation*}
Seeger  proved in \cite{Seeger1996JAMS} that the two conditions
$D<\infty$ and $L<\infty$ imply weak type $(1,1)$ boundedness of $T$ whenever $\widehat K\in L_\infty(\R^d)$. 
The distinction
from a full translation regularity condition is essential: the difference
in $\eta(a)$ compares two points on the same ray, and the angular variable
is integrated rather than differentiated.

In this paper, we extend Seeger's result \cite{Seeger1996JAMS} to the operator-valued setting. Let $(\M,\tau)$ be a von Neumann algebra with a normal semifinite faithful
trace, and set
\[
\n=L_\infty(\R^d)\,\bar\otimes\,\M,
\qquad \varphi=\int_{\R^d}dx\otimes\tau.
\]
For $1\leq p<\infty$, the noncommutative $L_p$-space associated to  $(\mathcal N,\varphi)$ is denoted by $L_p(\mathcal N)$ with the norm given by
$$\|f\|_p=\Big[\int_{\R^d}\tau(|f(x)|^p)\Big]^{1/p}.$$
As usual, set $L_\infty(\mathcal N)=\mathcal N$ with its
operator norm.
Given a measurable operator $f$ affiliated to $\mathcal N$, the weak $L_1$-norm  is defined as
\[
\|f\|_{1,\infty}
=\sup_{\lambda>0}\lambda\,
\varphi\bigl(\chi_{(\lambda,\infty)}(|f|)\bigr).
\]
All these are standard notions in noncommutative analysis, and we refer the reader to \cite{JungeXu2007JAMS,FackKosaki1986PJM} for more detailed information.
In this framework, the seminal result is due to Parcet \cite{Parcet2009}, who
gave the first noncommutative Calder\'on-Zygmund  decomposition, and established the noncommutative weak type $(1,1)$ estimate for
kernels with Lipschitz regularity.
Since the notion of ``points'' is unavailable in the noncommutative setting, the stopping time argument was replaced by the Cuculescu projections \cite{Cuculescu1971} from noncommutative martingale theory. Parcet applied it to obtain a Calder\'on-Zygmund decomposition but with four terms (see Theorem \ref{thm:decomposition}):
$$f=g+b+g_{\mathrm{off}}+b_{\mathrm{off}}.$$
The additional terms $g_{\mathrm{off}}$ and $b_{\mathrm{off}}$ are called off-diagonal good and bad functions respectively. They do not enjoy the same nice properties as the diagonal parts $g$ and $b$, which pose a major obstacle to applying classical ideas or techniques, and substantial new ideas are required.
For instance, even under the Lipschitz regularity, to deal with $g_{\mathrm{off}}$, Parcet explored the $TT^*$ method and a new pseudo-localization technique; see also \cite{MeiParcet2009}. This method does not apply to the H\"ormander regularity \eqref{hormander}, so he left it as  an open problem in \cite[Remark~5.5]{Parcet2009} and mentioned it again in his ICM survey \cite{Parcet2026ICM}. It is worth mentioning that Parcet's result has found a striking application to the resolution of the Nazarov--Peller conjecture in operator perturbation theory; see \cite{CaspersPotapovSukochevZanin2019AmerJMath}. 


Subsequently,  Cadilhac, Conde-Alonso and Parcet \cite{CCP2022} constructed a new Calder\'on--Zygmund decomposition by modifying $b_{\mathrm{off}}$ such that $g_{\mathrm{off}}$ vanishes. This enables them to obtain endpoint results under an $L_2$-H\"ormander regularity condition, which is stronger than the normal one \eqref{hormander} but weaker than the Lipschitz regularity condition. Under the same $L_2$-condition, the authors in \cite[Theorem~1.1]{HongLaiXu2023MathAnn} proved the noncommutative maximal weak type $(1,1)$ estimate for truncations of singular integral operators.


The present paper treats the  class of singular integrals with kernels satisfying  the angular $L\log L$ and averaged radial Dini hypotheses introduced above by Seeger.
\begin{thm}\label{thm:main}
	Let $T$ be defined in (\ref{general case}) and  $d\geq2$. Suppose that
	$\widehat K=m\in L_\infty(\mathbb R^d)$, $D<\infty$ and $L<\infty$.
	Then $T$ extends to a weak type $(1,1)$ operator on $L_1(\mathcal N)$. More precisely, for every $f\in L_1(\mathcal N)$,
	\[
	\sup_{\lambda>0}\lambda\,
	\varphi\bigl(\mathbf1_{(\lambda,\infty)}(|Tf|)\bigr)
	\leq C_d\bigl(\|m\|_\infty^2+\Lambda_K\bigr)\|f\|_1,
	\]
	where the constant $C_d$ depends only on the dimension $d$.
\end{thm}
	For the homogeneous kernel $K(r\theta)=r^{-d}\Omega(\theta)$, it is easy to check
\[
V(\theta)=(\log2)|\Omega(\theta)|,
\qquad \eta(a)\leq C_da^{-1}\|\Omega\|_1,\quad a\geq2.
\]
\begin{cor}
Let $\Omega\in L\log L(S^{d-1})$ satisfy (\ref{cancellation}).
Then for every $f\in L_1(\n)$,
	\[
	\sup_{\lambda>0}\lambda\,
	\varphi\bigl(\mathbf1_{(\lambda,\infty)}(|T_\Omega f|)\bigr)
	\leq C_{d,\Omega}\|f\|_1.
	\]
\end{cor}

The proof of Theorem \ref{thm:main} follows Seeger's kernel decomposition and microlocal analysis, but requires entirely new ideas or techniques. Seeger's $TT^*$ method-dependent argument motivates us to explore Parcet's Calder\'on-Zygmund decomposition \cite{Parcet2009} rather than the new one in \cite{CCP2022}. Indeed, the latter seems more efficient in dealing with the H\"ormander condition but does not seem compatible with the $TT^*$ method; see Lemma \ref{badfunction}. 
We sketch below the argument and the new ideas and techniques.

\smallskip

(i) The diagonal parts have the same nice properties as in the classical setting. For the off-diagonal parts, in addition to the estimates in \cite{Parcet2009}, we establish new estimates of $g_{\mathrm{off}}$ and the sum $g_{\mathrm{off}}+b_{\mathrm{off}}$ in Lemma \ref{HWX19}. We believe that, besides their important role in the present paper, they have further applications in noncommutative Calder\'on-Zygmund theory.

\smallskip

(ii) The diagonal part could be handled in the same way as in Seeger's paper \cite{Seeger1996JAMS}. To deal with the off-diagonal part, we decompose the truncated kernel $K_j$ in a different way in \eqref{decomposition of kernel}. More precisely, we set
\begin{align*}
&	D^s=\{\theta:V(\theta)\leq2^{\epsilon B(s)}\},\\
& 	U_{j}^{B(s)}(r\theta)=
 \int_0^\infty K(\rho\theta)2^{\delta B(s)-j}
 \omega\bigl(2^{\delta B(s)-j}(r-\rho)\bigr)\,d\rho
\end{align*}
with $\omega$ being a bump function around the origin, and then define
$$H_j^{B(s)}=\textbf{1}_{D^{s}}U_{j}^{B(s)}, \quad
		R^{B(s)}_{j}=\textbf{1}_{S^{d-1}\setminus D^s}U^{B(s)}_{j},\quad
		S^{B(s)}_{j}=K_j-U^{B(s)}_{j}$$
such that $K_j=H_j^{B(s)}+R_{j}^{B(s)}+S^{B(s)}_{j}$.
Here $B(s)$ is the unique integer of the form $8\cdot2^{n}$ such that $B(s)\leq s<2B(s)$, which is different from Seeger's $B(s)=s$. This, together with the estimates \eqref{RB} and  \eqref{SB} for the kernels $R$ and $S$ respectively, is quite relevant to {handling} the off-diagonal part in {Section} \ref{offdia-part}.

Interestingly, in order to get the sharp result for $R$ and $S$, we have to explore the combination $g_{\mathrm{off}}+b_{\mathrm{off}}$ before taking norms in Lemma \ref{HWX19}; see Remark \ref{rk:sum of g and b}.

\smallskip

(iii) To deal with $H$, we decompose it further in the same way as Seeger, that is,
$$H^{B(s)}_j=\Gamma_j^{B(s)}+(H_j^{B(s)}-\Gamma_j^{B(s)}).$$

For the term associated to $\Gamma$ acting on $b_{\mathrm{off}}$, we use the $TT^*$ method. However, the above choice of $B(s)$ and noncommutativity force us to establish a more flexible $L_2$-estimate in Lemma \ref{badfunction} in place of \cite[Lemma~2.1]{Seeger1996JAMS}. The proof becomes more complicated. More precisely, since the bad blocks are operator corners involving two stopping scales, we
estimate their directional square expansions using noncommutative H\"older inequality, and split the scale pairs into $|i-j|\leq B(s)-n$ and
$|i-j|>B(s)-n$.
For the term associated to $\Gamma$ acting on $g_{\mathrm{off}}$, an almost orthogonality principle involving martingale differences and Fourier multipliers supplies the required
$L_2$-bound in Lemma \ref{mc:prop:paired}.

\smallskip

(iv) Regarding the kernel $H-\Gamma$, our choice of $B(s)$ again forces us to have a more flexible estimate in Lemma \ref{mc:thm:complement} extending \cite[Lemma~2.2]{Seeger1996JAMS}. This will be used to deal with the off-diagonal good and bad functions separately. To make it applicable, we show that $g_{\mathrm{off}}$ enjoys the same $L_1$-estimate as $b_{\mathrm{off}}$ in Lemma \ref{HWX19} (i).

\bigskip


The paper is organized as follows. Section~\ref{czdecom} recalls the
noncommutative Calder\'on--Zygmund decomposition and its basic
properties. Section~\ref{sec:microlocal} establishes the microlocal estimates
that form the core of this paper. In {Section~\ref{sec:proof-main}},
we combine these estimates with the kernel decomposition
to prove Theorem~\ref{thm:main}. In the {Appendix
, we show that the ideas and estimates developed here also apply to Bochner--Riesz
means at the critical index.

\bigskip

Throughout the paper, $C_d$ denotes a positive constant
depending only on the dimension $d$, whose value may
change from line to line.

\bigskip


{\bf Note:}
{With} more effort, we are able to prove a noncommutative maximal weak type $(1,1)$ estimate of the truncated operators with kernels satisfying Seeger's conditions. {This result will appear in a subsequent paper.}

\section{Noncommutative Calder{\'o}n-Zygmund decomposition}\label{czdecom}
Throughout the paper, the pair $(\M,\tau)$ denotes a semifinite von Neumann algebra equipped with a normal semifinite faithful trace $\tau$, and we use the tensor product von Neumann algebra  $\n=L_\infty(\R^d)\,\bar\otimes\,\M$ with its natural trace $\varphi(f)=\int_{\R^d}\tau(f(x))dx.$
Let $\Q$ be the collection of dyadic cubes in $\R^d$, that is
\[
\Q = \bigcup_{k \in \mathbb{Z}} \Q_k = \bigcup_{k \in \mathbb{Z}} \left\{ 2^{-k} \cdot [j_1, j_1 + 1) \times [j_2, j_2 + 1) \times \dots [j_d, j_d + 1) : (j_1, \dots, j_d) \in \mathbb{Z}^d \right\}.
\]
For an integer $k\in \Z$, let $\Q_k$ be the set of dyadic cubes of side length $2^{-k}$, and  let $\sigma_k$ be the $\sigma$-algebra generated by $\Q_k$. Set
\begin{equation*}
 \mathcal N_k=L_\infty(\mathbb R^d,\sigma_k)
                     \,\bar\otimes\,\mathcal M,\qquad
 \mathsf E_k f
   =\sum_{Q\in\mathcal Q_k} f_Q\textbf{1}_Q,
 \quad f_Q=\frac1{|Q|}\int_Qf(y)\,dy.
\end{equation*}
For simplicity, we write $f_k=\mathsf{E}_kf$ and $df_{k}=f_{k}-f_{k-1}$.

As in \cite{Parcet2009}, consider the following dense class in $L_1(\mathcal N)$
$$\mathcal N_{c,+} =\Big\{
f: \R^d \to \M\cap L_1(\M) \, \big| \ f \geq0, \
\mathrm{supp}_{\R^d} f \ \ \mathrm{is \
	compact} \Big\},$$
where
$\mathrm{supp}_{\R^d} f=\mathrm{supp}\|f\|_\M$ denotes the Euclidean support of $f$,  as opposed to its support as an operator affiliated to $\n$. Given $f \in \mathcal{N}_{c,+}$ and $\lambda>0$, it follows from \cite[Lemma~3.1]{Parcet2009} that there exists some $m_\lambda(f)\in\Z$ such  that $f_k\leq \lambda\1_{\mathcal{N}}$ for all $k\leq m_\lambda(f)$ , where $\1_{\mathcal{N}}$ is the identity operator in $\mathcal{N}$.  Recall that the
Cuculescu projections $(q_k)_k$ are defined as
\begin{equation*}
	q_k=\begin{cases}
		\textbf{1}_\n,&k\leq m_\lambda(f),\\
		\mathbf1_{[0,\lambda]}(q_{k-1}f_kq_{k-1}),&k>m_\lambda(f).
	\end{cases}
\end{equation*}
Define $p_k=q_{k-1}-q_k=\sum_{Q\in\mathcal Q_k}p_Q\mathbf1_Q$ for $k\in\Z$ so that
\[
\sum_kp_k=\textbf{1}_\n-q,\qquad q=\bigwedge_kq_k.
\]
Then by definition,
	\begin{equation}\label{eq:pkqk}
		q_kf_kq_k\leq\lambda q_k,\qquad \lambda p_k\leq p_kf_kp_k\leq2^d\lambda p_k.
	\end{equation}

We now present Parcet's noncommutative Calder{\'o}n-Zygmund decomposition \cite{Parcet2009}.
For any $Q\in\Q$ and $s\in\N$, denote by $\ell(Q)$ the side length of $Q$ and by $sQ$ the cube with the same center as $Q$ such that $\ell(sQ)=s\ell(Q)$. For $x\in\mathbb R^d$, $Q_{x,k}$ denotes the unique cube in $\Q_k$ containing $x$.
\begin{thm}\label{thm:decomposition}
	Let $f\in\mathcal N_{c,+}$, $\lambda>0$.  Define
	\begin{equation*}
		\zeta=\1_{\mathcal{N}}-\bigvee_{Q\in\mathcal Q}p_Q\mathbf1_{513Q}.
	\end{equation*}
	There is a decomposition
	\begin{equation*}
		f=g+b+h
	\end{equation*}
such that the following properties hold.
	\begin{enumerate}
		\item[\rm{(i)}]
		$\varphi(\textbf{1}_\n-\zeta)\leq513^d\lambda^{-1}\|f\|_1.$
		\item[{\rm{(ii)}}] $g=qfq+\sum_kp_kf_kp_k,\ \
			\|g\|_1\leq\|f\|_1,\ \  \|g\|_\infty\leq2^d\lambda.$

		\item[{\rm{(iii)}}]
		$b=\sum_kb_k$, where $b_k=p_k(f-f_k)p_k$. For $Q\in\mathcal Q_k$,
		$
		\int_Qb_k(y)\,dy=0;
		$
		and for $x,y$ with $y\in513Q_{x,k}$,
		$\zeta(x)b_k(y)\zeta(x)=0$. Moreover, $\sum_k\|b_k\|_1\leq 2\|f\|_1$.
		\item[{\rm{(iv)}}]
		$h=g_{\mathrm{off}}+b_{\mathrm{off}}$, where
		\begin{equation*}
			\begin{split}
				g_{\mathrm{off}}=\sum_{r\geq1}\sum_kg_{k,r}&=\sum_{r\geq1}\sum_k\big[p_kdf_{k+r}q_{k+r-1}
				+q_{k+r-1}df_{k+r}p_k\big],\\
				b_{\mathrm{off}}=\sum_{r\geq1}\sum_kb_{k,r}&=\sum_{r\geq1}\sum_k\big[p_k(f-f_{k+r})p_{k+r}+p_{k+r}(f-f_{k+r})p_k\big].
			\end{split}
		\end{equation*}
		\item[{\rm{(v)}}]
		For the off-diagonal good term,
		\begin{equation*}
			\sup_{r\geq1}\sum_k\|g_{k,r}\|_2^2\leq C_d\lambda\|f\|_1.
		\end{equation*}
		\item[{\rm{(vi)}}]
		For the off-diagonal bad term,
		$$\sup_{r\geq1}\sum_k\|b_{k,r}\|_1\leq C_d\|f\|_1.$$
	\end{enumerate}
\end{thm}
The proof of Theorem \ref{thm:decomposition} can be found in \cite{Parcet2009} (see also \cite{CCP2022}). Here we find two additional estimates of the off-diagonal parts which may have further applications. In particular, they will play an instrumental role in the present paper.
\begin{lem}\label{HWX19}
The following two statements hold.
	\begin{enumerate}
		\item[\rm{(i)}] For the off-diagonal good term,
       $$\sup_{r\geq1}\sum_k\|g_{k,r}\|_1\leq C_d\|f\|_1.$$
\item[\rm{(ii)}] Let $h_{k,r}=g_{k,r}+b_{k,r}$ for $k\in\Z$ and $r\geq1$. For every integer $B\geq8$, define
\begin{equation*}
  h_j^B=\sum_{r\geq1}
 \sum_{\substack{n\geq8\\B\leq n+r<2B}}h_{n-j,r}.
\end{equation*}
Then
$\sum_j\|h_j^B\|_1\leq12B\|f\|_1.$
\end{enumerate}
\end{lem}

	\begin{proof}
(i) Fix $r\geq1$. Put
\[
a_{k,r}=p_k(f_{k+r}-f_{k+r-1})q_{k+r-1}.
\]
Then $g_{k,r}=a_{k,r}+a_{k,r}^*$. Since $p_kq_{k+r-1}=0$, one has $\|g_{k,r}\|_2^2=2\|a_{k,r}\|_2^2$. Furthermore, by the H\"older inequality and the fact that $q_{k+r-1}f_{\ell}q_{k+r-1}\leq C_d\lambda$ for $\ell\in\{k+r,k+r-1\}$, we have
$$\|p_kf_{\ell}q_{k+r-1}\|_2^2\leq\|p_kf_{\ell}p_k\|_1\|q_{k+r-1}f_{\ell}q_{k+r-1}\|_\infty\leq C_d\lambda\varphi(p_kfp_k).$$
Consequently,
\[
\|a_{k,r}\|_2^2\leq C_d\lambda\varphi(p_kfp_k).
\]
It then follows from the Cauchy-Schwarz inequality and the identity $p_ka_{k,r}=a_{k,r}$ that
\[
\begin{aligned}
	\sum_k\|g_{k,r}\|_1
	&\leq2\sum_k\varphi(p_k)^{1/2}\|a_{k,r}\|_2\\
	&\leq2\Big(\sum_k\varphi(p_k)\Big)^{1/2}
	\Big(\sum_k\|a_{k,r}\|_2^2\Big)^{1/2}
	\leq C_d\|f\|_1.
	 \end{aligned}
	 \]
	 This proves (i).

(ii) For an integer $s\geq8$, define
\[
 Z^{s}_j=\sum_{n=8}^{s-1}\sum_{r=1}^{s-n}h_{n-j,r},
\]
with $Z_{j}^{7}=Z_{j}^{8}=0$. Then reindexing the finite sums gives
\[
Z^{s}_j=\sum_{\substack{n\ge8,\ r\ge1\\n+r\le s}}h_{n-j,r},
\]
which implies that
$ h_j^B=Z_j^{2B-1}- Z_j^{B-1}.$

We claim that the proof will be finished if one can show
\begin{equation}\label{eq:triangular-budget11}
\sum_j\|Z_j^s\|_1\leq4(s-7)\|f\|_1,\qquad s\geq7.
\end{equation}
Indeed, this estimate implies
\[
 \sum_j\|h_j^B\|_1
 \leq4\bigl[(2B-8)+(B-8)\bigr]\|f\|_1
 \leq12B\|f\|_1.
\]

To prove (\ref{eq:triangular-budget11}), for integers $a\leq m$, set
\[
 Z(a,m)=\sum_{a\leq k<l\leq m}h_{k,l-k}.
\]
Then by definition,   $Z_j^s=Z(8-j,s-j)$. Hence, we are reduced to showing
\begin{equation}\label{eq:triangular-budget}
 \|Z(a,m)\|_1\leq4\sum_{k=a}^m\varphi(p_kfp_k).
\end{equation}
In fact, since $\sum_jp_j\leq\textbf{1}_{\n}$, it follows that
\[
 \sum_j\|Z_j^s\|_1\leq
 4\sum_j\sum_{n=8}^{s}\varphi(p_{n-j}fp_{n-j})\leq4(s-7)\|f\|_1.
\]

It remains to show (\ref{eq:triangular-budget}).  Observe that
\begin{equation}\label{eq:triangular-identity}
 \begin{aligned}
 Z(a,m)={}&P(a,m)fP(a,m)-\sum_{k=a}^mp_kfp_k
 +P(a,m)f_mq_m+q_mf_mP(a,m),
 \end{aligned}
\end{equation}
where $P(a,m)=\sum_{k=a}^m p_k$. To see this, since $p_\ell=q_{\ell-1}-q_\ell$,
$$a_{k,\ell}:=p_k(f-f_\ell)p_\ell+p_k(f_\ell-f_{\ell-1})q_{\ell-1}
=p_kfp_\ell+p_kf_\ell q_\ell-p_kf_{\ell-1}q_{\ell-1}.$$
Fix $k$. We see that
\[
\sum_{\ell=k+1}^m
\bigl(p_kf_\ell q_\ell-p_kf_{\ell-1}q_{\ell-1}\bigr)
=p_kf_mq_m-p_kf_kq_k=p_kf_mq_m,
\]
where  we have used $p_kf_kq_k=0$.
This gives that
\begin{align*}
	Z(a,m)
	& =\sum_{a\le k<\ell\le m}(a_{k,\ell}+a^*_{k,\ell})\\
	& =\sum_{a\le k<\ell\le m}(p_kfp_\ell+p_\ell fp_k)+P(a,m)f_mq_m
+q_mf_mP(a,m),
\end{align*}
which completes the proof of (\ref{eq:triangular-identity}) since
\[
P(a,m)fP(a,m)
=\sum_{k=a}^mp_kfp_k
+\sum_{a\le k<l\le m}(p_kfp_l+p_lfp_k).
\]

On the one hand,
\[
 \|P(a,m)fP(a,m)\|_1+\Big\|\sum_{k=a}^mp_kfp_k\Big\|_1
=2\sum_{k=a}^m\varphi(p_kfp_k).
\]
On the other hand, by using the fact
$q_mf_mq_m\leq\lambda q_m$ and the H\"older inequality, we obtain
\begin{align*}
\|P(a,m)f_mq_m\|_2^2
	& \leq\|P(a,m)f_m^{1/2}\|_2^2
       \|f_m^{1/2}q_m\|_\infty^2\\
	& \leq\lambda\varphi(P(a,m)f_mP(a,m))
 =\lambda\varphi(P(a,m)fP(a,m)).
\end{align*}
As a consequence,
\begin{align*}
\|P(a,m)f_mq_m\|_1
	& \leq\varphi(P(a,m))^{1/2}\|P(a,m)f_mq_m\|_2\\
	& \leq[\lambda\varphi(P(a,m))\varphi(P(a,m)fP(a,m))]^{1/2}\\
 &\leq\varphi(P(a,m)fP(a,m))=\sum_{k=a}^m\varphi(p_kfp_k),
\end{align*}
where the last inequality follows from
\[
\lambda\varphi(P(a,m))
\le\sum_{k=a}^m\varphi(p_kfp_k)
=\varphi(P(a,m)fP(a,m)).
\]
Obviously, the adjoint term $q_mf_mP(a,m)$ satisfies the same  bound. This finishes the proof.
\end{proof}


\section{Intermediate estimates}\label{sec:microlocal}
Let $\{H_j\}$ be a family of functions with
\begin{equation*}
\text{supp}H_j\subset\{x:2^{j-2}\leq |x|\leq2^{j+2}\}.
\end{equation*}
We assume that each $H_j(r\theta)$ is smooth in $r$ and that the estimates 
\begin{equation*}
	\sup_{0\leq\ell\leq N}\sup_j r^{d+\ell}|\partial_r^\ell H_j(r\theta)|\leq \mathfrak M_N
\end{equation*}
hold uniformly in $\theta$ and $r$.

The following construction is the same as that in Seeger's paper \cite[Section~2]{Seeger1996JAMS}. For an integer $s\geq8$, choose a family \(\{e^s_{\nu}:\nu\in\Theta_s\}\) of unit vectors in \(S^{d-1}\) satisfying
\begin{enumerate}
	\item [(a)] $|e^s_{\nu}-e^s_{\nu^\prime}|\geq2^{-s-10}d^{-1}$ whenever $\nu\neq\nu^\prime$;

	\item [(b)] For each $\theta\in S^{d-1}$, there exists an  $e^s_{\nu}$ such that $
	|\theta-e^s_{\nu}|\leq2^{-s-1}.
	$
\end{enumerate}
Indeed, one can take a maximal set satisfying (a); then the condition (b) holds automatically by maximality. We may construct a finite measurable partition
\(\{E^s_{\nu}\}_\nu\) of \(S^{d-1}\) with $e^s_{\nu}\in E^s_{\nu}$. Observe that
$\#\Theta_s\leq C_d2^{s(d-1)}$; moreover
$$H_j(x)=\sum_{\nu}H_{j\nu}^s(x),\qquad H_{j\nu}^s(x)=H_j(x)\textbf{1}_{E_\nu^s}(x/|x|).$$

Let \(\psi\in C_c^\infty([-4,4])\) be real, $0\leq\psi\leq1$
and $\psi=1$ on $[-2,2]$. Define the operator $P_\nu^s$ by
$$	\widehat{P_\nu^s\eta}(\xi)=\pi_\nu^s(\xi)\widehat \eta(\xi),\qquad \pi_\nu^s(\xi)=\psi\!\left(2^{s/2}
\frac{\langle\xi,e_\nu^s\rangle}{|\xi|}\right).$$

The basic splitting is
\begin{equation*}
	H_j=\Gamma_j^s+(H_j-\Gamma_j^s),
\end{equation*}
where $\Gamma_j^s=\sum_\nu P_\nu^sH_{j\nu}^s$.

We now present three lemmas, which are crucial to the use of noncommutative Calderón-Zygmund decomposition in the next section. In the following, we fix $f\in\n_{c,+}$ and $\lambda>0$.
The first lemma is an improvement  of \cite[Lemma~2.1]{Seeger1996JAMS}, which also works for the off-diagonal bad function.
\begin{lem}\label{badfunction}
	Let $n, s\geq8$ and $r\geq0$ be such that $8\leq s\leq n+r$.
	For each $j\in\Z$, set
	$X_j=p_{n-j}F_jp_{n+r-j}$, where $F_j\in\{f,f_{n+r-j}\}$. Then
	\[
	\Big\|\sum_{j}\Gamma_j^s*X_j\Big\|_2
	\leq C_d\sqrt{\lambda\|f\|_1}\sqrt{1+(s-n)_+}\,\mathfrak M_0\,2^{-s/4}.
	\]
The same bound holds with all $X_j$ replaced by $X_j^*$.
\end{lem}

The second $L_2$-estimate involves the off-diagonal  good function, which did not appear in Seeger's paper.
\begin{lem}\label{mc:prop:paired}
Let $n\geq8$, $r\geq1$ and $s\geq8$ be integers. Then
\[
 \Big\|\sum_{j}\Gamma_j^s*g_{n-j,r}\Big\|_2
 \leq C_d\mathfrak M_0\sqrt{\lambda\|f\|_1}\,2^{-s/4}.
\]
\end{lem}

The last lemma is a minor improvement of \cite[Lemma~2.2]{Seeger1996JAMS}, which is applicable to both the off-diagonal good and bad functions.

	\begin{lem}\label{mc:thm:complement}
	Let $j\in\mathbb Z$, $s\geq8$ and $n\in\mathbb Z$ be such that $s\leq n+1$.
    	Suppose that $F_Q\in L_1(\mathcal N)$ is supported in a cube $Q$
	with side length at most $2^{j-n}$ and satisfies
	$\int_{Q}F_Q(x)\,dx=0$. Then for $N>d$,
	\[
	\|(H_j-\Gamma_j^s)*F_Q\|_1
	\leq C_{d,N}
	\left[
	\mathfrak M_0\,2^{-s/4}
	+\mathfrak M_N\,2^{(d-N/4)s}
	\right]\|F_Q\|_1.
	\]
\end{lem}

We first prove Lemma \ref{mc:thm:complement} and postpone the proofs of Lemmas \ref{badfunction} and \ref{mc:prop:paired} to the next two subsections.
\begin{proof}[Proof of Lemma \ref{mc:thm:complement}.]
We follow Seeger's proof \cite[Lemma~2.2]{Seeger1996JAMS}.
    Let $\beta\in C^\infty(\R^d\setminus\{0\})$ be radial, supported in \(\{1/2\leq|\xi|\leq2\}\) such that $\sum_k\beta^2(2^{-k}\xi)=1$ for all $\xi\neq0$.  Denote $\mathcal L_k=\F^{-1}\beta_k$.  Consider the multiplier
\[
 m_{j\nu}^{sk}(\xi)
 =\beta_k(\xi)(1-\pi_\nu^s(\xi))\widehat {H^s_{j\nu}}(\xi).
\]
It then follows that
\[
 (H_j-\Gamma_j^s)*F_Q=\sum_\nu\sum_{k}
             \mathcal F^{-1}[m_{j\nu}^{sk}]*\mathcal L_k*F_Q.
\]

Using the cancellation of $F_Q$ and $\|\nabla \mathcal L_k\|_1\leq C_d2^k$, we obtain
\[
 \|\mathcal F^{-1}[m_{j\nu}^{sk}]*\mathcal L_k*F_Q\|_1
 \leq C_d\|\mathcal F^{-1}[m_{j\nu}^{sk}]\|_1
           \min\{1,2^{j+k-n}\}\|F_Q\|_1.
\]
Estimates (2.6) and (2.7) in \cite{Seeger1996JAMS} provide
\[
 \|\mathcal F^{-1}[m_{j\nu}^{sk}]\|_1
 \leq C_d\mathfrak M_0\,2^{-s(d-1)}
\]
and \[
 \|\mathcal F^{-1}[m_{j\nu}^{sk}]\|_1
 \leq C_{d,N}\mathfrak M_N\,2^{-s(d-1)}
       2^{-N(j+k-s/2)}
       \bigl(2^{d(j+k)}+2^{ds/2}\bigr)
\]
respectively.

On the one hand, we use Young's inequality to get
	\begin{align*}
	\sum_\nu\sum_{k:j+k\leq3s/4}\|\mathcal F^{-1}[m_{j\nu}^{sk}]*\mathcal L_k*F_Q\|_1
	& \leq	C_d\mathfrak M_0\,2^{-s(d-1)}\|F_Q\|_1\sum_\nu\sum_{k:j+k\leq3s/4}2^{j+k-n}\\
	& \leq 	C_d\mathfrak M_0\,2^{3s/4-n}\|F_Q\|_1\leq C_d\mathfrak M_0\,2^{-s/4}\|F_Q\|_1,
\end{align*}
where in the second inequality we also used the fact that $\#\Theta_s\leq C_d2^{s(d-1)}$, while the last inequality is a consequence of $n\geq s-1$.

On the other hand, using  $\|\mathcal L_k*F_Q\|_1\leq C_d\|F_Q\|_1$ and
	$\sum_\nu\sigma(E_\nu^s)=\sigma(S^{d-1})$, we obtain
	\begin{align*}
	\sum_\nu\sum_{k:j+k>3s/4}\|\mathcal F^{-1}[m_{j\nu}^{sk}]*\mathcal L_k*F_Q\|_1
	& \leq	C_{d,N}\mathfrak M_N\|F_Q\|_1
		2^{Ns/2}\sum_{k:j+k>3s/4}2^{-(N-d)(j+k)}\\
	& \leq 	C_{d,N}\mathfrak M_N\,2^{(d-N/4)s}\|F_Q\|_1.
\end{align*}
This finishes the proof by these two estimates.
\end{proof}

\subsection{Proof of Lemma \ref{badfunction}}


\begin{proof}[Proof of Lemma \ref{badfunction}.]
	By \cite{Seeger1996JAMS}, we know that
	\begin{equation}\label{bd1232}
	\sup_{\xi\neq0}\sum_\nu|\pi_\nu^s(\xi)|^2
	\leq C_d2^{s(d-3/2)}.
\end{equation}
Using  Plancherel's theorem, the Cauchy-Schwarz inequality and (\ref{bd1232}), we get
	\begin{align*}
		\Big\|\sum_{j}\Gamma_j^s*X_j\Big\|_2^2
		& \le\sup_{\xi\neq0}\sum_\nu |\pi_\nu^s(\xi)|^2\sum_\nu\Big\|\sum_{j}H_{j\nu}^{s}*X_j\Big\|_2^2\\
		& \le C_d2^{s(d-3/2)}\sum_\nu\Big\|\sum_{j}H_{j\nu}^{s}*X_j\Big\|_2^2.
	\end{align*}
We claim that it is enough to show 	for every $\nu$,
	\begin{equation}\label{eq:corner-principal}
		\Big\|\sum_{j}H_{j\nu}^{s}*X_j\Big\|_2^2
		\leq C_d\lambda(1+(s-n)_+)\mathfrak M_0^2
		2^{-2s(d-1)}\|f\|_1.
	\end{equation}
Indeed, with (\ref{eq:corner-principal}) in hand and noting $\#\Theta_s\leq C_d2^{s(d-1)}$, we have
\begin{align*}
	\Big\|\sum_{j}\Gamma_j^s*X_j\Big\|_2^2
	& \leq C_d2^{s(d-3/2)}2^{s(d-1)}\lambda (1+(s-n)_+)\mathfrak M_0^2 2^{-2s(d-1)}\|f\|_1\\ 	& =C_d\lambda\mathfrak M_0^22^{-s/2}(1+(s-n)_+)\|f\|_1.
\end{align*}

The following is devoted to the proof of (\ref{eq:corner-principal}).  Fix $\nu$ and write $Z_j=H_{j\nu}^{s}*X_j$.  Let $\varrho=(s-n)_+$.
We may decompose
\begin{equation}\label{eq:corner-principal11}
\Big\|\sum_jZ_j\Big\|_2^2=\sum_{j}\sum_{i:|i-j|\leq \varrho}\langle Z_i,Z_j\rangle+2\operatorname{Re}
\sum_j\sum_{\substack{i:i<j\\\ j-i>\varrho}}
\langle Z_i,Z_j\rangle.
\end{equation}

Consider the first term on the right-hand side of (\ref{eq:corner-principal11}).  By the Cauchy-Schwarz inequality,
\begin{align*}
\big|
\sum_{j}\sum_{i:|i-j|\leq \varrho}\langle Z_i,Z_j\rangle
\big|
&\le\frac12\sum_{|i-j|\le \varrho}
\bigl(\|Z_i\|_2^2+\|Z_j\|_2^2\bigr)\\
&\le(2\varrho+1)\sum_j\|Z_j\|_2^2.
\end{align*}
Hence in this case, it suffices to estimate $\sum_j\|Z_j\|_2^2$. Applying the H\"older inequality in Hilbert-valued noncommutative $L_p$-spaces (see e.g. \cite[Proposition~1.1]{Mei2007MemAMS}), one has for $x\in\R^d$,
\[
\|Z_j(x)\|_{L_2(\mathcal M)}^2
\le
\tau\bigl((|H_{j\nu}^s|*p_{n-j}F_jp_{n-j})(x)\bigr)
\bigl\|(|H_{j\nu}^s|*p_{n+r-j}F_jp_{n+r-j})(x)\bigr\|_\infty.
\]
Observe that $H_{j\nu}^s$ is supported in a symmetric rectangle $\mathcal R_{j\nu}^s$  centered at 0 with $d-1$ short sides of length $C2^{j-s}$ and one long side of length $C2^j$, where the long side is parallel to $e_\nu^s$. As a consequence,
\begin{equation}\label{cone-correlation1}
|\mathcal R^s_{j\nu}|\leq C_d2^{jd}2^{-s(d-1)}\qquad
|H_{j\nu}^s(z)|
\le C_d\mathfrak M_0\,2^{-jd}\mathbf1_{\mathcal R^s_{j\nu}}(z).
\end{equation}
 It follows from (\ref{cone-correlation1}) that
\begin{align*}
(|H_{j\nu}^s|*p_{n+r-j}F_jp_{n+r-j})(x)
&\le C_d\mathfrak M_0\,2^{-jd}\int_{x+\mathcal R^s_{j\nu}}(p_{n+r-j}F_jp_{n+r-j})(y)dy\\
&\le C_d\mathfrak M_0\,2^{-jd}\sum_{Q\in\Q_{n+r-j}\atop Q\cap\{x+\mathcal R^s_{j\nu}\}\neq\emptyset}\int_{Q}p_{n+r-j}f_{n+r-j}p_{n+r-j}(y)dy\\
&\le C_d\mathfrak M_0\,2^{-jd}\lambda\sum_{Q\in\Q_{n+r-j}\atop Q\cap\{x+\mathcal R^s_{j\nu}\}\neq\emptyset}|Q|\\
&\le  C_d\mathfrak M_0\,2^{-jd}\lambda|x+C_d\mathcal R^s_{j\nu}|\leq C_d\lambda\mathfrak M_02^{-s(d-1)},
\end{align*}
where the third inequality is due to $s\leq n+r$, and thus $2^{j-n-r}\leq2^{j-s}$.

On the other hand, since $\|H_{j\nu}^s\|_1\leq C_d\mathfrak M_02^{-s(d-1)}$, one can deduce from  Young's inequality that
\begin{align*}
\tau\bigl((|H_{j\nu}^s|*p_{n-j}F_jp_{n-j})(x)\bigr)
&\le \|H_{j\nu}^s\|_1\tau\bigl((p_{n-j}F_jp_{n-j})(x)\bigr)\\
&\le C_d\mathfrak M_02^{-s(d-1)}\tau\bigl((p_{n-j}fp_{n-j})(x)\bigr).
\end{align*}
Therefore, by these two estimates, we obtain
\begin{align*}
\sum_j\|Z_j\|_2^2
&= \sum_j\int_{\R^d}\|Z_j(x)\|_{L_2(\mathcal M)}^2dx\\
&\le  C_d\lambda \mathfrak M_0^22^{-2s(d-1)}\int_{\R^d}\sum_j\tau\bigl((p_{n-j}fp_{n-j})(x)\bigr)dx\\
&\leq C_d\lambda \mathfrak M_0^22^{-2s(d-1)}\|f\|_1,
\end{align*}
which implies that
\begin{equation}\label{cone-correlation113112}
\big|
\sum_{j}\sum_{i:|i-j|\leq \varrho}\langle Z_i,Z_j\rangle
\big|\leq C_d(1+\varrho)\lambda \mathfrak M_0^22^{-2s(d-1)}\|f\|_1.
\end{equation}

We now turn to estimating the second term in (\ref{eq:corner-principal11}).
For $i<j$ with $j-i>\varrho$, we have
\begin{align*}
\langle Z_i,Z_j\rangle
&=\iint
(\widetilde {H_{i\nu}^s}*H_{j\nu}^s)(y-z)
\,\tau(X_i(y)^*X_j(z))\,dy\,dz,
\end{align*}
where  $\widetilde {H_{i\nu}^s}(x)=\overline{H_{i\nu}^s(-x)}$. We may choose the symmetric rectangle $\mathcal R^s_{j\nu}$ large enough such that $\widetilde {H_{i\nu}^s}*H_{j\nu}^s$ is supported in $\mathcal R^s_{j\nu}$ for all $i\leq j$. Then
Young's inequality gives
\begin{align*}
\big|(\widetilde {H_{i\nu}^s}*H_{j\nu}^s)(y-z)\big|
&\leq \|\widetilde {H_{i\nu}^s}\|_1\|H_{j\nu}^s\|_\infty\textbf{1}_{\mathcal R^s_{j\nu}}(z-y)\\
&\le  C_d \mathfrak M_0^22^{-s(d-1)}2^{-jd}\textbf{1}_{\mathcal R^s_{j\nu}}(z-y).
\end{align*}

On the other hand,  by the Cauchy-Schwarz inequality and the basic inequality $ab\leq(a^2+b^2)/2$ for $a,b>0$, we see that
\begin{align*}
|\tau(X_i(y)^*X_j(z))|
&\le
\tau(p_{n+r-i}(y)F_i(y)p_{n+r-i}(y)p_{n+r-j}(z)F_j(z)p_{n+r-j}(z))^{\frac{1}{2}}
\\
&\ \ \times \tau(p_{n-i}(y)F_i(y)p_{n-i}(y)p_{n-j}(z)F_j(z)p_{n-j}(z))^{\frac{1}{2}}\\
&\le \frac{1}{2}\tau(p_{n+r-i}(y)F_i(y)p_{n+r-i}(y)p_{n+r-j}(z)F_j(z)p_{n+r-j}(z))\\
&\ \ +  \frac{1}{2}\tau(p_{n-i}(y)F_i(y)p_{n-i}(y)p_{n-j}(z)F_j(z)p_{n-j}(z))\\
&\ :=\frac{1}{2}[\tau(E_{i,j}(y,z))+\tau(G_{i,j}(y,z))].
\end{align*}

By the property of conditional expectation $\mathsf E_{n-i}(p_{n-i}F_ip_{n-i})=p_{n-i}f_{n-i}p_{n-i}$, one can deduce from (\ref{eq:pkqk}) that
\begin{align*}
\int_{z+\mathcal R^s_{j\nu}}p_{n-i}(y)F_i(y)p_{n-i}(y)dy
&\leq \sum_{Q\in\Q_{n-i}\atop Q\cap\{z+\mathcal R^s_{j\nu}\}\neq\emptyset}\int_{Q}(p_{n-i}f_{n-i}p_{n-i})(y)dy\\
&\le  C_d\lambda \sum_{Q\in\Q_{n-i}\atop Q\cap\{z+\mathcal R^s_{j\nu}\}\neq\emptyset}\int_{Q}p_{n-i}(y)dy.
\end{align*}
Since in this case $2^{i-n}<2^{j-n-\varrho}\leq2^{j-n}$ and $\sum_ip_i\leq\textbf{1}_{\n}$, one has
\begin{align*}
\sum_{\substack{i:i<j\\\ j-i>\varrho}}\int_{z+\mathcal R^s_{j\nu}}p_{n-i}(y)F_i(y)p_{n-i}(y)dy
&\leq  C_d\lambda \sum_{Q\in\Q_{n-i}\atop Q\cap\{z+\mathcal R^s_{j\nu}\}\neq\emptyset}\int_{Q}\sum_{\substack{i:i<j\\\ j-i>\varrho}}p_{n-i}(y)dy\\
&\le  C_d\lambda \sum_{Q\in\Q_{n-i}\atop Q\cap\{z+\mathcal R^s_{j\nu}\}\neq\emptyset}|Q|\\
&\le C_d\lambda2^{jd}2^{-s(d-1)}.
\end{align*}
Moreover, note also $2^{i-(n+s)}<2^{j-\varrho-(n+r)}\leq2^{j-n}$. The identical argument implies that
$$\sum_{\substack{i:i<j\\\ j-i>\varrho}}\int_{z+\mathcal R^s_{j\nu}}p_{n+r-i}(y)F_i(y)p_{n+r-i}(y)dy\le C_d\lambda2^{jd}2^{-s(d-1)}.$$
Therefore, by these two observations, we deduce by Fubini's theorem that
\begin{align*}
\Big|\sum_{\substack{i:i<j\\\ j-i>\varrho}}
\langle Z_i,Z_j\rangle\Big|
&\le  C_d \mathfrak M_0^22^{-s(d-1)}2^{-jd}\int_{\R^d}\sum_{\substack{i:i<j\\\ j-i>\varrho}}\int_{z+\mathcal R^s_{j\nu}}[\tau(E_{i,j}(y,z))+\tau(G_{i,j}(y,z))]dydz\\
&\le C_d \lambda\mathfrak M_0^22^{-2s(d-1)}\Big[\int_{\R^d}\tau(p_{n+r-j}(z)F_j(z)p_{n+r-j}(z))dz\\
&\ \ \ \ \  \ \ \ \ \ \ +\int_{\R^d}\tau(p_{n-j}(z)F_j(z)p_{n-j}(z))dz\Big]\\
&= C_d \lambda\mathfrak M_0^22^{-2s(d-1)}\Big[\int_{\R^d}\tau(p_{n+r-j}(z)f(z)p_{n+r-j}(z))dz\\
&\ \ \ \ \  \ \ \ \ \ \ +\int_{\R^d}\tau(p_{n-j}(z)f(z)p_{n-j}(z))dz\Big].
\end{align*}
Consequently, we find
\begin{equation}\label{cone-correlation113112w}
\Big|
2\operatorname{Re}
\sum_j\sum_{\substack{i:i<j\\\ j-i>\varrho}}
\langle Z_i,Z_j\rangle
\Big|\leq C_d\lambda \mathfrak M_0^22^{-2s(d-1)}\|f\|_1.
\end{equation}
This finishes the proof of (\ref{eq:corner-principal}) by (\ref{eq:corner-principal11}), (\ref{cone-correlation113112})  and (\ref{cone-correlation113112w}).
\end{proof}

\subsection{Proof of Lemma \ref{mc:prop:paired}}
The proof of Lemma \ref{mc:prop:paired} depends on the almost orthogonality principle, which should be well-known to the experts (see e.g. \cite[Lemma~3.7]{HongXu2021JFA}).
\begin{lem}\label{lem:almost-orthogonality}
	Let $\mathcal S_k$ be a bounded linear map on $L_2$ for each
	$k\in\mathbb Z$, and let $\eta\in L_2$.
	If $(u_n)_{n\in\mathbb Z}$ and $(v_n)_{n\in\mathbb Z}$
	are two sequences of functions in $L_2$ such that
	\[
	\eta=\sum_{n\in\mathbb Z}u_n,
	\qquad
	\sum_{n\in\mathbb Z}\|v_n\|_2^2<\infty,
	\]
	then
	\[
	\sum_k\|\mathcal S_k\eta\|_2^2
	\leq w^2\sum_{n\in\mathbb Z}\|v_n\|_2^2,
	\]
	provided that there is a sequence
	$(\sigma(j))_{j\in\mathbb Z}$ of positive numbers with
	\[
	w=\sum_{j\in\mathbb Z}\sigma(j)<\infty
	\]
	such that
	\[
	\|\mathcal S_ku_n\|_2\leq\sigma(n-k)\|v_n\|_2
	\]
	for every $n,k$.
\end{lem}

Now we are ready to prove Lemma \ref{mc:prop:paired}.
\begin{proof}[Proof of Lemma \ref{mc:prop:paired}.]
Using the same argument as in proving Lemma \ref{badfunction}, we have
	\begin{align*}
		\Big\|\sum_{j}\Gamma_j^s*g_{n-j,r}\Big\|_2^2
		 \le C_d2^{s(d-3/2)}\sum_\nu\Big\|\sum_{j}H_{j\nu}^{s}*g_{n-j,r}\Big\|_2^2.
	\end{align*}
By noting $\#\Theta_s\leq C_d2^{s(d-1)}$, it suffices to show
\begin{equation*}
 \Big\|\sum_{j}H^s_{j\nu}*g_{n-j,r}\Big\|_2^2
 \leq C_d\mathfrak M_0^2 2^{-2s(d-1)}\lambda\|f\|_1.
\end{equation*}

For $k\in\Z$, recall that $\mathcal L_k=\F^{-1}\beta_k$  was defined in the proof of Lemma \ref{mc:thm:complement}. Then by Plancherel's theorem,
$$\Big\|\sum_{j}H^s_{j\nu}*g_{n-j,r}\Big\|_2^2=\sum_k\Big\|\mathcal L_k*\sum_{j}H^s_{j\nu}*g_{n-j,r}\Big\|_2^2.$$
Applying Lemma \ref{lem:almost-orthogonality} to $\mathcal S_k=\mathcal L_k$, we are reduced to showing
	\begin{equation}\label{eq:good-almost-orthogonality}
		\begin{aligned}
			\|\mathcal L_k*(H_{j\nu}^s*g_{n-j,r})\|_2
			\leq
			C_d\mathfrak M_0\,2^{-s(d-1)}
			\sigma(n-j+r-k)\|g_{n-j,r}\|_2,
		\end{aligned}
	\end{equation}
	where
	\[
	\sigma(k)=
	\begin{cases}
		2^{-k},&k\geq0,\\
		2^{k/2},&k<0.
	\end{cases}
	\]
Note that
	\[
	\|\widehat {H^s_{j\nu}}\|_\infty
	\leq\|H^s_{j\nu}\|_1
	\leq C_d\mathfrak M_0\,2^{-s(d-1)}.
	\]
Then we use Plancherel's theorem to obtain
\begin{align*}
		\|\mathcal L_k*(H_{j\nu}^{s}*g_{n-j,r})\|_2
		&\leq \|\widehat {H^s_{j\nu}}\|_\infty\|\mathcal L_k*g_{n-j,r}\|_2\\
		&\leq C_d\mathfrak M_0\,2^{-s(d-1)}
		\|\mathcal L_k*g_{n-j,r}\|_2.
\end{align*}
Hence,  the proof of (\ref{eq:good-almost-orthogonality}) will follow if we can show
\begin{equation}\label{eq:good-frequency-decay}
\|\mathcal L_k*g_{n-j,r}\|_2\leq C_d\sigma(n-j+r-k)\|g_{n-j,r}\|_2.
\end{equation}

Suppose first that $k\leq n-j+r$. Recall that $\mathsf E_{n-j+r-1}g_{n-j,r}=0$. Then for every $u\in L_2(\n)$ with $\|u\|_2\leq1$, one has
\begin{align*}
|\langle \mathcal L_k*g_{n-j,r},u\rangle|
&=|\langle g_{n-j,r},\widetilde {\mathcal L_k}*u\rangle|\\
&=|\langle g_{n-j,r},
(I-\mathsf E_{n-j+r-1})(\widetilde {\mathcal L_k}*u)\rangle|\\
&\leq\|g_{n-j,r}\|_2
\|(I-\mathsf E_{n-j+r-1})(\widetilde {\mathcal L_k}*u)\|_2.
\end{align*}
Here $\widetilde {\mathcal L_k}(x)=\overline{\mathcal L_k(-x)}$.  By Jensen's inequality, the mean value theorem and the fact that $\|\nabla\mathcal L_k\|_1\leq C_d2^k$, we have
\begin{align*}
	\|(I-\mathsf E_{n-j+r-1})(\widetilde {\mathcal L_k}*u)\|^2_2
		&\leq \frac{1}{|Q|}\int_Q\int_Q\|(\widetilde {\mathcal L_k}*u)(x-y)-(\widetilde {\mathcal L_k}*u)(x)\|_{L_2(\M)}^2dxdy\\
		&\leq \frac{\ell(Q)^2}{|Q|}|Q|\|\nabla\mathcal L_k\|_1^2\|u\|_2^2\leq C_d2^{2k-2(n-j+r)}\|u\|_2^2.
\end{align*}
Taking the supremum over $\|u\|_2\leq1$, we obtain
\begin{equation}\label{eq:good-translation111}
\|\mathcal L_k*g_{n-j,r}\|_2
\leq C_d2^{k-(n-j+r)}\|g_{n-j,r}\|_2.
\end{equation}

Suppose next that $k>n-j+r$. 	The identity
	$\mathsf E_{n-j+r}g_{n-j,r}=g_{n-j,r}$
	implies that $g_{n-j,r}$ is a constant operator on dyadic cubes of
	side $2^{-(n-j+r)}$. Hence
	\begin{equation}\label{eq:good-translation}
		\|g_{n-j,r}(\cdot-y)-g_{n-j,r}\|_2
		\leq C_d
		\min\{1,(2^{n-j+r}|y|)^{1/2}\}\|g_{n-j,r}\|_2.
	\end{equation}
Indeed, this estimate is obvious by the triangle inequality whenever $|y|>2^{-(n-j+r)}$; on the other hand, if
$|y|\leq2^{-(n-j+r)}$, 	denote by $\mathcal D_y$ the set of points $x$ such that $x$ and
	$x-y$ belong to different cubes of $\mathcal Q_{n-j+r}$. 	For each $Q\in\mathcal Q_{n-j+r}$,
    $$	|\mathcal D_y\cap Q|\leq C_d2^{n-j+r}|y|\,|Q|.$$
This gives that
  	\begin{align*}
		\int_{\mathcal D_y}\|g_{n-j,r}(x)\|_{L_2(\mathcal M)}^2\,dx
	\leq C_d2^{n-j+r}|y|\,\|g_{n-j,r}\|_2^2.
	\end{align*}
The same bound also holds for $\|g_{n-j,r}(\cdot-y)\|_2$.  Hence, (\ref{eq:good-translation}) is proved.

	Since $\int\mathcal L_k=0$, we obtain
	\[
	\mathcal L_k*g_{n-j,r}(x)
	=
	\int \mathcal L_k(y)
	\bigl[g_{n-j,r}(x-y)-g_{n-j,r}(x)\bigr]\,dy.
	\]
	Using \eqref{eq:good-translation} and
	$\mathcal L_k(y)=2^{kd}\check\beta(2^ky)$, we find
	\begin{align*}
		\|\mathcal L_k*g_{n-j,r}\|_2
		&\leq C_d\|g_{n-j,r}\|_2
		\int |\mathcal L_k(y)|(2^{n-j+r}|y|)^{1/2}\,dy\\
		&=C_d2^{-(k-(n-j+r))/2}\|g_{n-j,r}\|_2
		\int|\check\beta(y)||y|^{1/2}\,dy\\
		&\leq C_d2^{-(k-(n-j+r))/2}\|g_{n-j,r}\|_2,
	\end{align*}
    where the last inequality holds since $\beta$ is a Schwartz function.
 This, combined with (\ref{eq:good-translation111}), proves \eqref{eq:good-frequency-decay}, and hence the proof is complete.
\end{proof}

\section{Proof of Theorem \ref{thm:main}}\label{sec:proof-main}
With the preparatory work in the previous sections, we are now able to conclude the proof of Theorem \ref{thm:main}.

 Using the standard density argument and splitting into the sum of four positive functions, it suffices to show the weak type $(1,1)$ estimate for elements from $\mathcal N_{c,+}$. Now fix $f\in \mathcal N_{c,+}$ and $\lambda\in(0,+\infty)$.
By Theorem \ref{thm:decomposition}, we can decompose $f$ as $f=g+b+h$.
Therefore, it suffices to show
\begin{align*}
	\varphi(\chi_{(\lambda,\infty)}(|T\phi|))\leq C_d\bigl(\|m\|_\infty^2+\Lambda_K\bigr)\frac{\| f\|_{1}}{\lambda},
\end{align*}
where $\phi\in\{g,b,h\}$. By Plancherel's theorem and
the H\"{o}lder inequality,
\begin{equation*}
	\varphi(\chi_{(\lambda,\infty)}(|Tg|))
	\leq\lambda^{-2}\|m\|_\infty^2\|g\|_2^2
	\leq C_d\lambda^{-1}\|m\|_\infty^2\|f\|_1.
\end{equation*}
For $\phi\in\{b,h\}$, we decompose $T\phi$ as
$$T\phi=\zeta^\perp T\phi \zeta^\perp + \zeta T\phi\zeta^\perp
+ \zeta^\perp T\phi\zeta + \zeta T\phi
\zeta.$$
By the distribution inequality and Theorem \ref{thm:decomposition} (i),
\begin{equation*}
	\begin{split}
		\varphi\big(\chi_{(\lambda,\infty)}(|T\phi|)\big)\leq&\ \varphi (\1_\mathcal{N}-\zeta) +
		\varphi \big(\chi_{(\lambda/4,\infty)}(|\zeta T\phi\zeta|)\big)\\
		\leq&\ C_d\lambda^{-1}\|f\|_{1}+\varphi \big(\chi_{(\lambda/4,\infty)}(|\zeta T\phi\zeta|)
		\big).
	\end{split}
\end{equation*}
Hence, it remains to deal with $\zeta T\phi\zeta$.
The estimate for the case $\phi=b$ follows essentially from Seeger's argument \cite{Seeger1996JAMS}, given in {Section} \ref{diag-bad}. The novel part lies in the case $\phi=h$, which will be carried out in {Section} \ref{offdia-part}.




	\subsection{Dyadic decomposition and kernel estimates}  Recall from the proof of Lemma \ref{mc:thm:complement} that $\beta\in C^\infty(\R^d\setminus\{0\})$ is radial, supported in \(\{1/2\leq|\xi|\leq2\}\) such that $\sum_k\beta^2(2^{-k}\xi)=1$ for all $\xi\neq0$. Put $\chi(r)=\beta(re_1)^2$ for $r>0$. Since $\beta$ is radial,
	this choice is independent of the unit vector $e_1$, and
	\begin{equation*}
		K(x)=\sum_{j\in\mathbb Z}K_j(x)\quad(x\neq0),\qquad
		K_j(r\theta)=\chi(2^{-j}r)K(r\theta).
	\end{equation*}

	Set $\mathscr B=\{8\cdot2^n:n\geq0\}$. For every integer $s\geq8$,
there is a unique $B(s)\in\mathscr B$ with
$B(s)\leq s<2B(s)$. Define
\begin{equation}\label{eq:block-parameters}
	D^s=\{\theta\in S^{d-1}:V(\theta)\leq2^{\delta B(s)}\|V\|_{1}\}.
\end{equation}
	Let $\omega\in C_c^\infty((0,2^{-10}))$ have integral
one. Define
\[
	U_{j}^{B(s)}(r\theta)=
 \int_0^\infty K(\rho\theta)2^{\delta B(s)-j}
 \omega\bigl(2^{\delta B(s)-j}(r-\rho)\bigr)\,d\rho,
\]
where $\delta=\frac1{100(d+2)}$. Then
\begin{align}\label{decomposition of kernel}
K_j=H_j^{B(s)}+R_{j}^{B(s)}+S^{B(s)}_{j},
\end{align}
where
$$H_j^{B(s)}=\textbf{1}_{D^{s}}U_{j}^{B(s)}, \quad
		R^{B(s)}_{j}=\textbf{1}_{S^{d-1}\setminus D^s}U^{B(s)}_{j},\quad
		S^{B(s)}_{j}=K_j-U^{B(s)}_{j}.$$

The above decomposition follows from Seeger's idea \cite{Seeger1996JAMS} but with a different choice of $B(s)$. Below,
the estimates \eqref{eq:regularized-seminorms}, \eqref{fd:eq:radial-error-per-block} and \eqref{eq:scalar-errors} are standard (see \cite[pp.~103--104]{Seeger1996JAMS}); but \eqref{RB} and \eqref{SB} are new variants, which follow from our choice of $B(s)$ and will play an important role in handling the off-diagonal parts in {Section} \ref{offdia-part}.
\begin{lem}\label{kernelestimate1}
For integers $s\geq8$ and  $N>0$, define
$$ \mathfrak M_N(B(s))=\sup_j\esssup_{\theta\in S^{d-1}}
 \sup_{\rho>0}\max_{0\leq a\leq N}
 \rho^{d+a}|\partial_\rho^aH_j^{B(s)}(\rho\theta)|.$$
Then
\begin{equation}\label{eq:regularized-seminorms}
 \mathfrak M_N(B(s))\leq C_{d,N}\|V\|_1\,2^{(N+2)\delta B(s)}.
\end{equation}
\end{lem}

\begin{lem}\label{kernelestimate3}
	For an integer $s\geq8$, set $\gamma(B(s))=\sup_j\|R_{j}^{B(s)}\|_1$. Then
	\begin{equation}\label{fd:eq:radial-error-per-block}
		\gamma(B(s))\leq C_d\int_{\{V(\theta)>2^{\delta B(s)}\|V\|_1\}}V(\theta)\,d\sigma(\theta).
	\end{equation}
	As a consequence,
	\begin{align}\label{RB}
\sum_{s\geq8}\gamma(B(s))=\sum_{B\in\mathscr B}B\gamma(B)\leq C_d(1+L).\end{align}
\end{lem}
\begin{proof}
It suffices to prove \eqref{RB}. The identity can be verified easily.
It then follows from Fubini's theorem and (\ref{fd:eq:radial-error-per-block}) that
\begin{align*}
	\sum_{B\in\mathscr B}B\gamma(B)
	\leq C_d\int_{S^{d-1}}V(\theta)
	\sum_{B\in\mathscr B}B\,
	\textbf{1}_{\{V(\theta)/\|V\|_1>2^{\delta B}\}}
	\,d\sigma(\theta).
\end{align*}
Set $a=V(\theta)/\|V\|_1$. If $a>2^{\delta B}$, then $B<\frac{\log a}{\delta\log2}$.  Let $B_*\in\mathscr B$ be the largest scale satisfying this condition.
Then
\[
\sum_{\substack{B\in\mathscr B\\a>2^{\delta B}}}B
=8+16+\cdots+B_*
=2B_*-8
<\frac{2\log a}{\delta\log2}.
\]
Consequently,
$$\sum_{B\in\mathscr B}B\gamma(B)\leq C_d\int_{S^{d-1}}V(\theta)
\log_+\!\left(\frac{V(\theta)}{\|V\|_1}\right)
	\,d\sigma(\theta)\leq C_d(1+L),$$
as desired.
\end{proof}

\begin{lem}\label{kernelestimate2}
For every integer $s\geq8$, define $e(B(s))=\sup_j\|S_{j}^{B(s)}\|_1$. Then
\begin{equation}\label{eq:scalar-errors}
e(B(s))\leq C_dW,\quad \forall s\geq8; \quad
 e(B(s))\leq C_d\eta(2^{\delta B(s)-3}),\quad \forall B(s)\geq 10\delta^{-1}.
\end{equation}
	As a consequence,
	\begin{align}\label{SB}
  \sum_{s\geq8}e(B(s))=\sum_{{B\in\mathscr B}}B\,e(B)\leq C_d(W+D).
  \end{align}
\end{lem}
\begin{proof}

We just prove \eqref{SB}.
On the one hand,
\[
\sum_{\substack{B\in\mathscr B\\B\leq10/\delta}}B\,e(B)
\leq C_dW
\sum_{\substack{B\in\mathscr B\\B\leq10/\delta}}B
\leq C_dW.
\]

For the remaining part, the second estimate in
\eqref{eq:scalar-errors} gives
\[
e(B)\leq C_d\eta(2^{\delta B-3}),
\qquad B>10/\delta.
\]
Since $\eta$ is nonincreasing,
\[
\int_{2^{\delta B/2-3}}^{2^{\delta B-3}}
\eta(a)\,\frac{da}{a}
\geq
\eta(2^{\delta B-3})\,\frac{\delta B\log2}{2}.
\]
Consequently,
\[
B\,e(B)
\leq \frac{2C_d}{\delta\log2}
\int_{2^{\delta B/2-3}}^{2^{\delta B-3}}
\eta(a)\,\frac{da}{a}.
\]
These intervals lie in $[2,\infty)$ and have disjoint interiors because $B>10/\delta$.
Therefore,
\begin{align*}
\sum_{\substack{B\in\mathscr B\\B>10/\delta}}B\,e(B)
&\leq \frac{2C_d}{\delta\log2}
\sum_{\substack{B\in\mathscr B\\B>10/\delta}}
\int_{2^{\delta B/2-3}}^{2^{\delta B-3}}
\eta(a)\,\frac{da}{a}\\
&\leq C_d\delta^{-1}
\int_2^\infty\eta(a)\,\frac{da}{a}
\leq C_dD.
\end{align*}
This completes the proof.
\end{proof}

{
\begin{remark}
In \cite{Seeger1996JAMS}, Seeger's original definition of $D^s$ in \eqref{eq:block-parameters} was given by taking $B(s)= s$, and the resulting estimates seem not useful since the off-diagonal parts do not satisfy estimates similar to those for the diagonal bad part. Our choice of $B(s)$ and the resulting estimates \eqref{RB}, \eqref{SB} are used to obtain the new estimate of the off-diagonal part \eqref{off-part} in Section \ref{offdia-part}.
\end{remark}
}

\subsection{Estimate for diagonal bad function}\label{diag-bad}
As previously noted, the proof of this part is essentially the same as Seeger's argument \cite{Seeger1996JAMS}.
Recall that $b=\sum_kb_k$, where $b_k=p_k(f-f_k)p_k$. For $x\in\R^d$, we may write
\begin{align*}
	\zeta(x)Tb(x)\zeta(x)
	& =\zeta(x)\sum_{j\in\Z}\int_{\R^d}K_j(x-y)b(y)dy\zeta(x)\\
	& =\zeta(x)\sum_{n}\sum_{j\in\Z}\int_{\R^d}K_j(x-y)b_{n-j}(y)dy\zeta(x).
\end{align*}

For $Q\in\mathcal Q_{n-j}$ and $y\in Q$, $K_j(x-y)$ vanishes
unless $|x-y|\leq2^{j+1}$. If $n<8$, then $\zeta(x)p_Q=0$ whenever $x\notin
513Q$ since $|x-y|>2^8\,2^{j-n}\geq2^{j+1}$.
Thus
$$\zeta(x)Tb(x)\zeta(x)=\zeta(x)\sum_{n\geq8}\sum_{j\in\Z}\int_{\R^d}K_j(x-y)b_{n-j}(y)dy\zeta(x).$$
Now fix $n\geq8$ and $j\in\Z$. Decompose the kernel $K_j$ as
\begin{equation*}
K_j=\Gamma_j^{B(n)}+\big(H^{B(n)}_{j}-\Gamma_j^{B(n)}\big)+R^{B(n)}_{j}+S^{B(n)}_{j}.
\end{equation*}

By Lemma \ref{badfunction} with $s=B(n)$ and (\ref{eq:regularized-seminorms}) with $N=0$,
	\begin{equation*}
	\begin{split}
		\Big\|\sum_{n\geq8}\sum_{j}
	\Gamma_j^{B(n)}*b_{n-j}\Big\|_2
		&\leq  C_d\sqrt{\lambda\|f\|_1}\sum_{n\geq8}2^{-\alpha B(n)}\\
		&\leq  C_d\sqrt{\lambda\|f\|_1}\sum_{n\geq8}2^{-\alpha n/2}\leq  C_d\sqrt{\lambda\|f\|_1},
	\end{split}
\end{equation*}
where $\alpha=1/4-2\delta$.

Since $\int_Qb_{n-j}(x)dx=0$ for $Q\in\Q_{n-j}$, it follows from Lemma \ref{mc:thm:complement}, Lemma \ref{kernelestimate1} with $N=8(d+1)$, the fact that $B(n)\geq n/2$ and Theorem \ref{thm:decomposition} (iii) that
	\begin{equation*}
	\begin{split}
		\Big\|\sum_{n\geq8}\sum_{j}\big(H^{B(n)}_{j}-\Gamma_j^{B(n)}\big)*b_{n-j}\Big\|_1
		&\leq C_d\sum_{n\geq8}(2^{-\alpha n/2}+2^{-\epsilon n/2})
		\sum_{j}\|b_{n-j}\|_1\\
		&\leq C_d\|f\|_1,
	\end{split}
\end{equation*}
where $\epsilon=N/4-d-(N+2)\delta>0$.

On the other hand, by Lemma \ref{kernelestimate3} and Theorem \ref{thm:decomposition} (iii),
	\begin{equation*}
	\begin{split}
		\Big\|\sum_{n\geq8}\sum_{j}R^{B(n)}_{j}*b_{n-j}\Big\|_1
		&\leq \sum_{n\geq8}\sum_{j}\|R^{B(n)}_{j}\|_1\|b_{n-j}\|_1\\
		&\leq \|f\|_1\sum_{n\geq8}\gamma(B(n))\leq C_d(1+L)\|f\|_1.
	\end{split}
\end{equation*}

The same argument for $S^{B(n)}_{j}$, using Lemma \ref{kernelestimate2}, gives
\begin{equation*}
	\begin{split}
		\Big\|\sum_{n\geq8}\sum_{j}S^{B(n)}_{j}*b_{n-j}\Big\|_1
		&\leq2\|f\|_1\sum_{n\geq8}e(n)
		\leq C_d(D+W)\|f\|_1.
	\end{split}
\end{equation*}

Finally, combining the preceding kernel decomposition with the four bounds above, we deduce that
	\begin{equation*}
	\begin{split}
		\varphi\big(\chi_{(\lambda,\infty)}(|\zeta Tb\zeta|)\big)
		&\leq \lambda^{-2}\Big\|\sum_{n\geq8}\sum_{j}
	\Gamma_j^{B(n)}*b_{n-j}\Big\|^2_2\\
		&\ \ \ \ +\lambda^{-1}\Big\|\sum_{n\geq8}\sum_{j}\big(R^{B(n)}_{j}+S^{B(n)}_{j}\big)*b_{n-j}\Big\|_1\\
        &\leq C_d(1+D+W+L)\|f\|_1.
	\end{split}
\end{equation*}

\subsection{Estimate for off-diagonal part}\label{offdia-part}
Recall that
\begin{align*}
h_{k,r}=g_{k,r}+b_{k,r}.
\end{align*}
The same argument as for the diagonal part gives
\begin{equation}\label{eq:off-far-reduction}
	\zeta Th\zeta
	=\zeta\Big(\sum_{n\geq8}\sum_{r\geq1}\sum_{j}
	K_j*h_{n-j,r}\Big)\zeta.
\end{equation}
In this part, we decompose
\begin{align*}
K_j=\Gamma_j^{B(n+r)}+\big(H^{B(n+r)}_{j}-\Gamma_j^{B(n+r)}\big)+R^{B(n+r)}_{j}+S^{B(n+r)}_{j}.
\end{align*}

\subsubsection{Estimates related to the kernels $\Gamma_j$ and $H_j-\Gamma_j$}
In this case, we handle $b_{\mathrm{off}}$ and $g_{\mathrm{off}}$ separately.
Applying Lemma \ref{badfunction}, Lemma \ref{kernelestimate1} and Theorem \ref{thm:decomposition} (vi), we obtain
	\begin{equation*}
	\begin{split}
		\Big\|\sum_{n\geq8}\sum_{r\geq1}\sum_{j}
		\Gamma_j^{B(n+r)}*b_{n-j,r}\Big\|_2
		&\leq C_d\sqrt{\lambda\|f\|_1}
		\sum_{n\geq8}\sum_{r\geq1}\sqrt{r+1}\,2^{-\alpha B(n+r)}\\
		&\leq C_d\sqrt{\lambda\|f\|_1}.
	\end{split}
\end{equation*}

On the other hand, it follows from Lemma \ref{mc:thm:complement} with $s=n+r$ and Lemma \ref{kernelestimate1}  that
	\begin{equation*}
	\begin{split}
		&\Big\|\sum_{n\geq8}\sum_{r\geq1}\sum_{j}
		\big(H^{B(n+r)}_{j}-\Gamma_j^{B(n+r)}\big)*b_{n-j,r}\Big\|_1\\
		&\quad\leq C_d\sum_{n\geq8}\sum_{r\geq1}
		(2^{-\alpha B(n+r)}+2^{-\epsilon B(n+r)})
		\sum_j\|b_{n-j,r}\|_1
		\leq C_d\|f\|_1.
	\end{split}
\end{equation*}

We now turn to estimating $g_{\mathrm{off}}$.  By Lemma~\ref{mc:prop:paired}, Lemma \ref{kernelestimate1} and Theorem \ref{thm:decomposition} (v),
	\begin{equation*}
	\begin{split}
		\Big\|\sum_{n\geq8}\sum_{r\geq1}\sum_{j}
		\Gamma_j^{B(n+r)}*g_{n-j,r}\Big\|_2
		&\leq C_d\sqrt{\lambda\|f\|_1}
		\sum_{n\geq8}\sum_{r\geq1}2^{-\alpha B(n+r)}\\
		&\leq C_d\sqrt{\lambda\|f\|_1}.
	\end{split}
\end{equation*}

On the other hand, recall that
$\mathsf{E}_{n-j+r-1}g_{n-j,r}=0$, which satisfies the conditions in Lemma~\ref{mc:thm:complement}.  Hence, we use Lemma~\ref{mc:thm:complement},  Lemma \ref{kernelestimate1} and Lemma \ref{HWX19} (i) to deduce that
\begin{equation*}
	\begin{split}
		&\Big\|\sum_{n\geq8}\sum_{r\geq1}\sum_{j}
		\big(H^{B(n+r)}_{j}-\Gamma_j^{B(n+r)}\big)*g_{n-j,r}\Big\|_1\\
		&\quad\leq C_d\sum_{n\geq8}\sum_{r\geq1}
		(2^{-\alpha B(n+r)}+2^{-\epsilon B(n+r)})
		\sum_j\|g_{n-j,r}\|_1
		\leq C_d\|f\|_1.
	\end{split}
\end{equation*}

\subsubsection{Estimates related to the kernels $R_j$ and $S_j$}
In this case, we explore the summation $h_{k,r}=b_{k,r}+g_{k,r}$.
Recall that for $B\in\mathscr B$,
$$h_j^B=\sum_{r\geq1}
 \sum_{\substack{n\geq8\\B\leq n+r<2B}}h_{n-j,r}.$$
For each fixed $r\geq1$, write
$$
\sum_{n\geq8}=\sum_{B\in\mathscr B} \sum_{\substack{n\geq8\\ B\leq n+r<2B}}.$$ Then
one can deduce from the definition of $h_j^B$ that
\begin{equation*}
 \begin{aligned}
 & \sum_{n\geq8}\sum_{r\geq1}\sum_{j}
 R_j^{B(n+r)}*h_{n-j,r}
=\sum_{B\in\mathscr B}\sum_{j}R_j^{B}*h_j^B,
\\
 &\sum_{n\geq8}\sum_{r\geq1}\sum_{j}
 S_j^{B(n+r)}*h_{n-j,r}
 =\sum_{B\in\mathscr B}\sum_{j}S_j^{B}*h_j^B.
 \end{aligned}
\end{equation*}
Therefore, by Young's inequality, Lemma~\ref{HWX19} (ii) and Lemma \ref{kernelestimate3}, we obtain

\begin{equation*}
 \Big\|\sum_{B\in\mathscr B}\sum_{j}R_j^{B}*h_j^B\Big\|_1
\leq12\|f\|_1\sum_{B\in\mathscr B}B\gamma(B)\leq C_d(1+L)\|f\|_1;
\end{equation*}
and similarly, Lemma \ref{kernelestimate2} gives
\begin{equation*}
 \Big\|\sum_{B\in\mathscr B}\sum_{j}S_j^{B}* h_j^B\Big\|_1
 \leq12\|f\|_1\sum_{B\in\mathscr B}Be(B)\leq C_d(W+D)\|f\|_1.
\end{equation*}

\begin{remark}\label{rk:sum of g and b}
In the literature
\cite{Parcet2009,Cadilhac2018,HongXu2021JFA},
proofs using Parcet's decomposition proceed by estimating
$Tg_{\mathrm{off}}$ and $Tb_{\mathrm{off}}$ separately.
In the present setting, however, the direct separate estimates
are insufficient under our assumptions, and we instead exploit
the  summation
$g_{\mathrm{off}}+b_{\mathrm{off}}$. Indeed, using
\[
\sum_j\|b_{n-j,r}\|_1+\sum_j\|g_{n-j,r}\|_1
\leq C_d\|f\|_1,
\]
Young's inequality gives
\begin{align*}
\Big\|\sum_{n\geq8}\sum_{r\geq1}\sum_j
R_j^{B(n+r)}*b_{n-j,r}\Big\|_1
&\leq C_d\|f\|_1
\sum_{B\in\mathscr B}\gamma(B)
\sum_{\substack{n\geq8,\ r\geq1\\B\leq n+r<2B}}1\\
&\leq C_d\|f\|_1
\sum_{B\in\mathscr B}B^2\gamma(B).
\end{align*}
This estimate leads to a stronger
condition. In fact,
\begin{align*}
\sum_{B\in\mathscr B}B^2\gamma(B)
&\leq C_d\int_{S^{d-1}}V(\theta)
\sum_{B\in\mathscr B}B^2
\mathbf1_{\{V(\theta)/\|V\|_1>2^{\delta B}\}}
\,d\sigma(\theta)\\
&\leq C_d\int_{S^{d-1}}V(\theta)
\left(\log_+\frac{V(\theta)}{\|V\|_1}\right)^2\,d\sigma(\theta).
\end{align*}

Similarly, by using the monotonicity of $\eta$,
\[
\sum_{B\in\mathscr B}B^2e(B)
\leq C_d\left[
W+\int_2^\infty\eta(a)\log a\,\frac{da}{a}
\right].
\]
Thus these direct separate estimates require an additional
logarithmic factor in both the angular integrability and the radial
Dini condition.
%
\end{remark}

Now \eqref{eq:off-far-reduction} and the estimates proved in this subsection give
\begin{equation}\label{off-part}
 \varphi\big(\chi_{(\lambda,\infty)}(|\zeta Th\zeta|)\big)
\leq C_d\lambda^{-1}(L+W+D)\|f\|_1,
\end{equation}
which finishes the proof for the off-diagonal part.


\appendix
\section{Bochner--Riesz means}\label{app:bochner-riesz}

In this appendix, we adapt the proof of Theorem~\ref{thm:main}
to establish the weak type $(1,1)$ estimate for noncommutative
Bochner-Riesz means at the critical index.
The explicit oscillation of the Bochner--Riesz kernel yields
the desired estimates in Lemma~\ref{kernelestimate223},
which simplify the $TT^*$ argument in
Lemma~\ref{badfunction}. The good off-diagonal terms are handled
by martingale cancellation, as in
Lemma~\ref{mc:prop:paired}. Consequently, the proof and
calculations are simpler than those of
Theorem~\ref{thm:main}, since the microlocal decomposition
used there is no longer needed.

Consider the operator
\begin{equation}\label{Bochner-Riesz}
	\widehat{\mathcal B f}(\xi)
	=(1-|\xi|^2)_+^\mu\widehat f(\xi),
	\qquad \mu=\frac{d-1}{2}.
\end{equation}

\begin{thm}\label{thm:critical-BR}
	Let $\mathcal B$ be defined by \eqref{Bochner-Riesz}.
	Then for every $f\in L_1(\mathcal N)$,
	\[
	\sup_{\lambda>0}\lambda\,
	\varphi\bigl(\chi_{(\lambda,\infty)}
	(|\mathcal B f|)\bigr)
	\leq C_d\|f\|_1.
	\]
\end{thm}

	Denote by $V_\mu$ the inverse Fourier transform of
$(1-|\xi|^2)_+^\mu$. 
	Choose a finite smooth partition of unity
$\{\varsigma_\nu\}_\nu$ on $S^{d-1}$, with each
$\varsigma_\nu$ supported in a sufficiently small spherical cap. Recall that $\chi$ is a smooth function supported in $[1/2,2]$ such that $\sum_j\chi(2^{-j}r)=1$ for $r>0$.
Following the reduction in \cite[Section~3]{Christ1988}, we obtain
\begin{equation}\label{Bochner-Riesz1}
	V_\mu=\sum_\nu\sum_{j>0}V_{j,\nu}+W,
\end{equation}
where
\begin{equation*}
	V_{j,\nu}(x)
	=c_d\varsigma_\nu(x/|x|)
	\cos\left(2\pi|x|-\frac{\pi d}{2}\right)
	\chi(2^{-j}|x|)|x|^{-d}.
\end{equation*}
Here $c_d$ is a fixed normalization constant, and the remainder $W$ satisfies
\[
|W(x)|\leq C_d(1+|x|)^{-d-1}.
\]
By Young's inequality, $\|W*f\|_1\leq C_d\|f\|_1$. Hence, it remains to estimate the first term on the right-hand side of (\ref{Bochner-Riesz1}). Since the partition $\{\varsigma_\nu\}_\nu$ is finite, it suffices to fix one element $\eta$ of the partition and prove that for every $\lambda>0$,
\begin{eqnarray}\label{Bochner-Riesz3}
\varphi\bigl(\chi_{(\lambda,\infty)}(\big|\sum_{j>0}V_j*f\big|)\bigr)
\le C_d\lambda^{-1}\|f\|_1,
\end{eqnarray}
where $$V_j(x)=c_d\eta(x/|x|)\cos\left(2\pi|x|-\pi d/2\right)\chi(2^{-j}|x|)|x|^{-d}.$$

The following two estimates are given in \cite[Lemma~3.1]{Christ1988}.
\begin{lem}\label{kernelestimate223}
	Let $\widetilde {V_j}(x)=\overline{V_j(-x)}$. Then
\begin{align*}
	|(\widetilde {V_j}*V_j)(x)|
	&\le C_d2^{-jd}(1+|x|)^{-\mu},\  j\ge1,
	\label{self}\\
	|(\widetilde {V_i}*V_j)(x)|
	&\le C_d2^{-jd}2^{-\mu i},
	\ 1\le i\leq j-4. 
\end{align*}
Both kernels above are supported in $\{|x|\leq C_d2^j\}$.
\end{lem}

We now prove (\ref{Bochner-Riesz3}). By the standard density argument and decomposition into four positive functions, it suffices to consider $f\in\n_{c,+}$. Fix such an $f$ and $\lambda>0$. By Theorem \ref{thm:decomposition}, we may decompose $f=g+b+g_{\mathrm{off}}+b_{\mathrm{off}}$.

The estimate for $g$ follows from Lemma \ref{kernelestimate223}. Indeed, by Plancherel's theorem,
\begin{eqnarray}\label{Bochner-Riesz58}
\big\|V_j*g\big\|_2\leq\|g\|_2\|V_j*\widetilde {V_j}\|^{1/2}_1\leq C_d2^{-j\mu/2}\|g\|_2,
\end{eqnarray}
which implies  that
\begin{eqnarray*}
\big\|\sum_{j>0}V_j*g\big\|_2\leq C_d\|g\|_2\sum_{j>0}2^{-j\mu/2}\leq C_d\sqrt{\lambda\|f\|_1}.
\end{eqnarray*}

\subsection{Estimates for the bad functions} Using the projection $\zeta$ introduced in Theorem \ref{thm:decomposition}, it suffices to estimate $$\zeta \sum_{j>0}V_j*(b+b_{\mathrm{off}})\zeta.$$
Set $b_{k,0}=b_{k}$. The same support argument as in Section~\ref{sec:proof-main} gives
$$	\zeta\sum_{j>0}V_j*(b+b_{\mathrm{off}})\zeta
=\zeta\sum_{n\geq8}\sum_{r\geq0}\sum_{j>0}
V_j*b_{n-j,r}\zeta.$$

	Fix $n\geq8$ and $r\geq0$. First consider the range $2j<n+r$.
We use the cancellation of $b_{n-j,r}$, Young's inequality and $\|\nabla V_j\|_1\leq C_d$ to get
\[
\|V_j*b_{n-j,r}\|_1
\leq C_d2^{j-n-r}\|\nabla V_j\|_1
\|b_{n-j,r}\|_1\leq C_d2^{j-n-r}
\|b_{n-j,r}\|_1,
\]
Since $\|b_{n-j,r}\|_1\leq\|f\|_1$, this yields
	\begin{equation*}
	\begin{split}
		\big\|\sum_{n\geq8}\sum_{r\geq0}\sum_{j:2j<n+r}
		V_j*b_{n-j,r}\big\|_1
		&\leq C_d\|f\|_1
		\sum_{n\geq8}\sum_{r\geq0}\sum_{0<j<(n+r)/2}2^{-j-n-r}\\
		&\leq C_d\|f\|_1
		\sum_{n\geq8}\sum_{r\geq0}2^{-(n+r)/2}\leq C_d\|f\|_1.
	\end{split}
\end{equation*}

For the remaining range, we use the following estimate. Its proof follows the argument of Lemma~\ref{badfunction}. We include the details for the restriction $2j\geq n+r$.
\begin{lem}\label{corner}
	Let $n\ge1$ and $r\ge0$ be integers. Then
	\begin{equation}\label{cornerbound}
		\Big\|\sum_{j:j\geq (n+r)/2}
		V_j*X_j\Big\|_2^2
		\le C_d\lambda\|f\|_1\,2^{-\mu (n+r)/4},
	\end{equation}
where $X_j$ is defined as in Lemma~\ref{badfunction}. The same estimate holds with $X_j$ replaced by $X_j^*$.
\end{lem}
\begin{proof}
Set
	$
	R_j=V_j*(p_{n-j}F_jp_{n+r-j}).
	$
	We decompose
	\[
	\Big\|\sum_jR_j\Big\|_2^2
	=
	\sum_j\sum_{|i-j|\leq3}\langle R_i,R_j\rangle
	+2\operatorname{Re}
	\sum_j\sum_{i\leq j-4}\langle R_i,R_j\rangle.
	\]
	
The first term on the right-hand side is at most $7\sum_j\|R_j\|_2^2$. For $i\leq j$, one has
\[
\langle R_i,R_j\rangle
=\iint(\widetilde V_i*V_j)(y-z)
\tau(X_i(y)^*X_j(z))\,dy\,dz.
\]
As in the proof of Lemma~\ref{badfunction}, we have
\begin{align*}
	|\tau(X_i(y)^*X_j(z))|
	&\leq
	\tau\bigl((p_{n-i}F_ip_{n-i})(y)
	(p_{n-j}F_jp_{n-j})(z)\bigr)^{1/2}\\
	&\quad\times
	\tau\bigl((p_{n+r-i}F_ip_{n+r-i})(y)
	(p_{n+r-j}F_jp_{n+r-j})(z)\bigr)^{1/2}.
\end{align*}
For $a=n$ or $a=n+r$, we have
\[
\int_Qp_{a-j}F_jp_{a-j}
\leq2^d\lambda|Q|,\qquad \ell(Q)=2^{j-a}.
\]
It follows that, for every
$\rho>0$ and $z\in\R^d$,
\[
\int_{|y-z|\leq\rho}p_{a-j}(y)F_j(y)p_{a-j}(y)\,dy
\leq C_d\lambda(\rho+2^{j-a})^d.
\]
Using the first estimate in Lemma~\ref{kernelestimate223}, the above inequality with $\rho=C_d2^j$ and the fact that $j\geq (n+r)/2$, we obtain
\begin{align*}
	&\iint|\widetilde V_j*V_j(y-z)|
	\tau\bigl((p_{a-j}F_jp_{a-j})(y)
	(p_{a-j}F_jp_{a-j})(z)\bigr)\,dy\,dz\\
	&\qquad\leq C_d\lambda(2^{-da}+2^{-\mu j})
	\varphi(p_{a-j}F_jp_{a-j})\leq C_d\lambda(2^{-d a}+2^{-\mu (n+r)/2})
	\varphi(p_{a-j}F_jp_{a-j}).
\end{align*}
Therefore, by the Cauchy--Schwarz inequality and the fact that $\mu\leq d$,
\[
\begin{aligned}
	\sum_j\|R_j\|_2^2
	&\leq C_d\lambda\|f\|_1
	\bigl(2^{-dn}+2^{-\mu(n+r)/2}\bigr)^{1/2}
	\bigl(2^{-d(n+r)}+2^{-\mu(n+r)/2}\bigr)^{1/2}\\
	&\leq C_d\lambda\|f\|_1\,2^{-\mu(n+r)/4}.
\end{aligned}
\]

For the second term, we use the second estimate in Lemma~\ref{kernelestimate223} and $2^{i-a}\leq2^j$ to get
\begin{align*}
	&\iint|\widetilde V_i*V_j(y-z)|
	\tau\bigl((p_{a-i}F_ip_{a-i})(y)
	(p_{a-j}F_jp_{a-j})(z)\bigr)\,dy\,dz\\
	&\qquad\leq C_d\lambda2^{-\mu i}
	\varphi(p_{a-j}F_jp_{a-j}).
\end{align*}
Hence
\[
\begin{aligned}
	&\sum_{\substack{j\geq(n+r)/2}}
	\sum_{\substack{(n+r)/2\leq i\leq j-4}}
	\iint|\widetilde V_i*V_j(y-z)|
	\tau\bigl((p_{a-i}F_ip_{a-i})(y)
	(p_{a-j}F_jp_{a-j})(z)\bigr)\,dy\,dz\\
	&\leq C_d\lambda
	\sum_{j\geq(n+r)/2}
	\varphi(p_{a-j}F_jp_{a-j})
	\sum_{i\geq(n+r)/2}2^{-\mu i}\leq C_d\lambda\|f\|_1\,2^{-\mu(n+r)/2}.
\end{aligned}
\]
Consequently, the Cauchy--Schwarz inequality gives
\[
	\left|2\operatorname{Re}
	\sum_j\sum_{i\leq j-4}\langle R_i,R_j\rangle\right|
	\leq C_d\lambda\|f\|_1\,2^{-\mu(n+r)/2}.
\]
This finishes the proof.
\end{proof}

	Applying Lemma~\ref{corner} with
$F_j=f$ and then with $F_j=f_{n+r-j}$, we deduce from the triangle inequality that
	\begin{equation*}
	\begin{split}
		\big\|\sum_{n\geq8}\sum_{r\geq0}\sum_{j:2j\geq n+r}
		V_j*b_{n-j,r}\big\|_2
		&\leq C_d\sqrt{\lambda\|f\|_1}
		\sum_{n\geq8}\sum_{r\geq0}\,
		2^{-\mu(n+r)/8}\\
		&\leq C_d\sqrt{\lambda\|f\|_1},
	\end{split}
\end{equation*}
which completes the estimates for the bad functions.

\subsection{Estimate for the off-diagonal good function} By the same support argument as above, it suffices to estimate
\[
\zeta\sum_{n\ge8}\sum_{r\ge1}\sum_{j>0}
V_j*g_{n-j,r}\zeta.
\]

Fix $n\geq8$ and $r\geq1$. Since
$\mathsf E_{n-j+r-1}g_{n-j,r}=0$, the same
argument as in the proof of Lemma~\ref{mc:prop:paired} gives
$$\|V_j*g_{n-j,r}\|_2
\leq C_d2^{j-n-r}\|\nabla V_j\|_1\|g_{n-j,r}\|_2\leq C_d2^{j-n-r}\|g_{n-j,r}\|_2.$$
Combining this with (\ref{Bochner-Riesz58}), we obtain
\begin{equation}\label{minimum}
	\|V_j*g_{n-j,r}\|_2\le C_d
		\min\{2^{-\mu j/2},2^{j-n-r}\}\|g_{n-j,r}\|_2.
\end{equation}

By (\ref{minimum}) and the Cauchy--Schwarz inequality,
\begin{align*}
	\Big\|\sum_{j>0}V_j*g_{n-j,r}\Big\|_2
	&\leq C_d
	\Big(
	\sum_{j>0}\min\{2^{-\mu j},2^{2(j-n-r)}\}
	\Big)^{1/2}
	\Big(\sum_j\|g_{n-j,r}\|_2^2\Big)^{1/2}\\
	&\leq C_d\sqrt{\lambda\|f\|_1}\,
	2^{-\mu(n+r)/4}.
\end{align*}
Summing over $n$ and $r$, we obtain
\begin{equation*}
	\begin{split}
		\big\|\sum_{n\ge8}\sum_{r\ge1}\sum_{j>0}
		V_j*g_{n-j,r}\big\|_2
		&\leq C_d\sqrt{\lambda\|f\|_1}
		\sum_{n\geq8}\sum_{r\geq1}\,
		2^{-\mu(n+r)/4}\\
		&\leq C_d\sqrt{\lambda\|f\|_1}.
	\end{split}
\end{equation*}
This completes the proof.

\begin{remark}
Note that almost simultaneously, the author in 
\cite{Lai2026BR} obtained the same estimate for Bochner-Riesz means. Below, we briefly explain the main differences in the argument. 
For the off-diagonal good functions, we combine martingale cancellation
directly with the dyadic $L_2$ decay \eqref{Bochner-Riesz58} to obtain
\eqref{minimum}, avoiding the Littlewood--Paley
decomposition in \cite[Section~5]{Lai2026BR}. For the bad functions, we
estimate operator corners using the size bounds for the composed kernels in
Lemma~\ref{kernelestimate223},
in place of the H\"older continuity estimates in
\cite[Section~6]{Lai2026BR}.
\end{remark}

\begin{remark}
By standard transference techniques, we also have the weak type $(1,1)$ estimates of Bochner-Riesz means with critical index on quantum Euclidean spaces, quantum tori, matrix algebras, as well as von Neumann algebras generated by discrete amenable groups admitting finite orthogonal cocycles; see e.g. \cite{CCP2022}. The similar results also hold for the class of Fourier multipliers considered in Theorem~\ref{thm:main}.
\end{remark}

\bigskip
\noindent\textbf{Acknowledgements.}
G. Hong was partially supported by the National Natural Science
Foundation of China (Grant Nos.~12325105, W2441002, and W2611005).
L. Wang was supported by a grant from the Research Grants Council of the
Hong Kong Administrative Region, China (Project No.~CityU 11308924).
X. Wang was supported by the National Natural Science Foundation of China
(Grant Nos.~12601242 and W2611005).
B. Xu was supported by the National Natural Science Foundation of China
(Grant No.~12501172) and the Fundamental Research Funds for the Central
Universities (Grant No.~20720250061).

\bigskip

\noindent\textbf{AI statement.}
The authors used ChatGPT for language polishing and \LaTeX{} editing. During exploratory mathematical discussions with GPT-6-Astra, the AI tool suggested the
decomposition in Lemma \ref{HWX19} (ii), a key ingredient of
the proof. They also helped polish the manuscript and
provided editorial feedback. The authors checked and
developed the suggested idea, supplied further arguments,
and take full responsibility for the mathematical content
of the paper.




\bibliographystyle{alpha}
\bibliography{main}

@article {Seeger1996JAMS,
    AUTHOR = {Seeger, Andreas},
     TITLE = {Singular integral operators with rough convolution kernels},
   JOURNAL = {J. Amer. Math. Soc.},
  FJOURNAL = {Journal of the American Mathematical Society},
    VOLUME = {9},
      YEAR = {1996},
    NUMBER = {1},
     PAGES = {95--105},
      ISSN = {0894-0347,1088-6834},
   MRCLASS = {42B20},
  MRNUMBER = {1317232},
MRREVIEWER = {Steve\ Hofmann},
       DOI = {10.1090/S0894-0347-96-00185-3},
       URL = {https://doi.org/10.1090/S0894-0347-96-00185-3},
}

@article{Parcet2009,
  author = {Parcet, Javier},
  title = {Pseudo-localization of singular integrals and noncommutative
              {C}alder\'{o}n-{Z}ygmund theory},
  journal = {J. Funct. Anal.},
  fjournal = {Journal of Functional Analysis},
  volume = {256},
  year = {2009},
  number = {2},
  pages = {509--593},
  issn = {0022-1236,1096-0783},
  mrclass = {42B20 (46L52 46L53)},
  mrnumber = {2476951},
  mrreviewer = {Fyodor\ A.\ Sukochev},
  doi = {10.1016/j.jfa.2008.04.007},
  url = {https://doi.org/10.1016/j.jfa.2008.04.007}
}

@article {FackKosaki1986PJM,
    AUTHOR = {Fack, Thierry and Kosaki, Hideki},
     TITLE = {Generalized {$s$}-numbers of {$\tau$}-measurable operators},
   JOURNAL = {Pacific J. Math.},
  FJOURNAL = {Pacific Journal of Mathematics},
    VOLUME = {123},
      YEAR = {1986},
    NUMBER = {2},
     PAGES = {269--300},
      ISSN = {0030-8730,1945-5844},
   MRCLASS = {46L10 (46L50 47A30)},
  MRNUMBER = {840845},
MRREVIEWER = {H.\ Araki},
       URL = {http://projecteuclid.org/euclid.pjm/1102701004},
}

@article {Cuculescu1971,
    AUTHOR = {Cuculescu, Ion},
     TITLE = {Martingales on von {N}eumann algebras},
   JOURNAL = {J. Multivariate Anal.},
  FJOURNAL = {Journal of Multivariate Analysis},
    VOLUME = {1},
      YEAR = {1971},
    NUMBER = {1},
     PAGES = {17--27},
      ISSN = {0047-259X},
   MRCLASS = {60B15},
  MRNUMBER = {295398},
MRREVIEWER = {H.\ Umegaki},
       DOI = {10.1016/0047-259X(71)90027-3},
       URL = {https://doi.org/10.1016/0047-259X(71)90027-3},
}

@article {Cadilhac2018,
    AUTHOR = {Cadilhac, L\'{e}onard},
     TITLE = {Weak boundedness of {C}alder\'{o}n-{Z}ygmund operators on
              noncommutative {$L_1$}-spaces},
   JOURNAL = {J. Funct. Anal.},
  FJOURNAL = {Journal of Functional Analysis},
    VOLUME = {274},
      YEAR = {2018},
    NUMBER = {3},
     PAGES = {769--796},
      ISSN = {0022-1236,1096-0783},
   MRCLASS = {42B25 (46L10)},
  MRNUMBER = {3734984},
MRREVIEWER = {Maria\ J.\ Carro},
       DOI = {10.1016/j.jfa.2017.11.003},
       URL = {https://doi.org/10.1016/j.jfa.2017.11.003},
}

@article {CCP2022,
    AUTHOR = {Cadilhac, L\'{e}onard and Conde-Alonso, Jos\'{e} M. and Parcet, Javier},
     TITLE = {Spectral multipliers in group algebras and noncommutative
              {C}alder\'{o}n-{Z}ygmund theory},
   JOURNAL = {J. Math. Pures Appl. (9)},
  FJOURNAL = {Journal de Math\'{e}matiques Pures et Appliqu\'{e}es.
              Neuvi\`eme S\'{e}rie},
    VOLUME = {163},
      YEAR = {2022},
     PAGES = {450--472},
      ISSN = {0021-7824,1776-3371},
   MRCLASS = {42B20 (42B15 46L52)},
  MRNUMBER = {4438906},
MRREVIEWER = {Xudong\ Lai},
       DOI = {10.1016/j.matpur.2022.05.011},
       URL = {https://doi.org/10.1016/j.matpur.2022.05.011},
}

@article {CZ1952,
    AUTHOR = {Calder\'{o}n, Alberto P. and Zygmund, Antoni S.},
     TITLE = {On the existence of certain singular integrals},
   JOURNAL = {Acta Math.},
  FJOURNAL = {Acta Mathematica},
    VOLUME = {88},
      YEAR = {1952},
     PAGES = {85--139},
      ISSN = {0001-5962,1871-2509},
   MRCLASS = {42.4X},
  MRNUMBER = {52553},
MRREVIEWER = {H.\ Kober},
       DOI = {10.1007/BF02392130},
       URL = {https://doi.org/10.1007/BF02392130},
}

@article {CZ1956,
    AUTHOR = {Calder\'{o}n, Alberto P. and Zygmund, Antoni S.},
     TITLE = {On singular integrals},
   JOURNAL = {Amer. J. Math.},
  FJOURNAL = {American Journal of Mathematics},
    VOLUME = {78},
      YEAR = {1956},
     PAGES = {289--309},
      ISSN = {0002-9327,1080-6377},
   MRCLASS = {44.0X},
  MRNUMBER = {84633},
MRREVIEWER = {H.\ Kober},
       DOI = {10.2307/2372517},
       URL = {https://doi.org/10.2307/2372517},
}

@article {Christ1988,
    AUTHOR = {Christ, Michael},
     TITLE = {Weak type {$(1,1)$} bounds for rough operators},
   JOURNAL = {Ann. of Math. (2)},
  FJOURNAL = {Annals of Mathematics. Second Series},
    VOLUME = {128},
      YEAR = {1988},
    NUMBER = {1},
     PAGES = {19--42},
      ISSN = {0003-486X,1939-8980},
   MRCLASS = {42B20},
  MRNUMBER = {951506},
MRREVIEWER = {Jos\'{e}\ Garc\'{\i}a-Cuerva},
       DOI = {10.2307/1971461},
       URL = {https://doi.org/10.2307/1971461},
}

@article {CR1988,
    AUTHOR = {Christ, Michael and Rubio de Francia, Jos\'e{} L.},
     TITLE = {Weak type {$(1,1)$} bounds for rough operators. {II}},
   JOURNAL = {Invent. Math.},
  FJOURNAL = {Inventiones Mathematicae},
    VOLUME = {93},
      YEAR = {1988},
    NUMBER = {1},
     PAGES = {225--237},
      ISSN = {0020-9910,1432-1297},
   MRCLASS = {42B20 (42B25)},
  MRNUMBER = {943929},
MRREVIEWER = {Takahide\ Kurokawa},
       DOI = {10.1007/BF01393693},
       URL = {https://doi.org/10.1007/BF01393693},
}

@article {Hofmann1988,
    AUTHOR = {Hofmann, Steve},
     TITLE = {Weak {$(1,1)$} boundedness of singular integrals with
              nonsmooth kernel},
   JOURNAL = {Proc. Amer. Math. Soc.},
  FJOURNAL = {Proceedings of the American Mathematical Society},
    VOLUME = {103},
      YEAR = {1988},
    NUMBER = {1},
     PAGES = {260--264},
      ISSN = {0002-9939,1088-6826},
   MRCLASS = {42B20 (47G05)},
  MRNUMBER = {938680},
MRREVIEWER = {Kenneth\ F.\ Andersen},
       DOI = {10.2307/2047563},
       URL = {https://doi.org/10.2307/2047563},
}

@article {Tao1999,
    AUTHOR = {Tao, Terence},
     TITLE = {The weak-type {$(1,1)$} of {$L\log L$} homogeneous convolution
              operator},
   JOURNAL = {Indiana Univ. Math. J.},
  FJOURNAL = {Indiana University Mathematics Journal},
    VOLUME = {48},
      YEAR = {1999},
    NUMBER = {4},
     PAGES = {1547--1584},
      ISSN = {0022-2518,1943-5258},
   MRCLASS = {42B20 (47G10)},
  MRNUMBER = {1757083},
MRREVIEWER = {Javier\ Duoandikoetxea},
       DOI = {10.1512/iumj.1999.48.1784},
       URL = {https://doi.org/10.1512/iumj.1999.48.1784},
}

@article {MeiParcet2009,
    AUTHOR = {Mei, Tao and Parcet, Javier},
     TITLE = {Pseudo-localization of singular integrals and noncommutative
              {L}ittlewood-{P}aley inequalities},
   JOURNAL = {Int. Math. Res. Not. IMRN},
  FJOURNAL = {International Mathematics Research Notices. IMRN},
      YEAR = {2009},
    NUMBER = {8},
     PAGES = {1433--1487},
      ISSN = {1073-7928,1687-0247},
   MRCLASS = {46L51 (42B20 46B25 46L52)},
  MRNUMBER = {2496770},
MRREVIEWER = {Christian\ Le Merdy},
       DOI = {10.1093/imrn/rnn165},
       URL = {https://doi.org/10.1093/imrn/rnn165},
}

@article {JungeXu2007JAMS,
	AUTHOR = {Junge, Marius and Xu, Quanhua},
	TITLE = {Noncommutative maximal ergodic theorems},
	JOURNAL = {J. Amer. Math. Soc.},
	FJOURNAL = {Journal of the American Mathematical Society},
	VOLUME = {20},
	YEAR = {2007},
	NUMBER = {2},
	PAGES = {385--439},
	ISSN = {0894-0347},
	MRCLASS = {46L53 (37A55 46B70 46L51 46L55 47A35 47L05)},
	MRNUMBER = {2276775},
	MRREVIEWER = {Adam G. Skalski},
	DOI = {10.1090/S0894-0347-06-00533-9},
	URL = {https://doi.org/10.1090/S0894-0347-06-00533-9},
}

@article {Mei2007MemAMS,
	AUTHOR = {Mei, Tao},
	TITLE = {Operator valued {H}ardy spaces},
	JOURNAL = {Mem. Amer. Math. Soc.},
	FJOURNAL = {Memoirs of the American Mathematical Society},
	VOLUME = {188},
	YEAR = {2007},
	NUMBER = {881},
	PAGES = {vi+64},
	ISSN = {0065-9266},
	MRCLASS = {46L52 (42B30 42B35 46B70 46M35)},
	MRNUMBER = {2327840},
	MRREVIEWER = {Christian Le Merdy},
	DOI = {10.1090/memo/0881},
	URL = {https://doi.org/10.1090/memo/0881},
}

@article {HongLaiXu2023MathAnn,
    AUTHOR = {Hong, Guixiang and Lai, Xudong and Xu, Bang},
     TITLE = {Maximal singular integral operators acting on noncommutative
              {$L_p$}-spaces},
   JOURNAL = {Math. Ann.},
  FJOURNAL = {Mathematische Annalen},
    VOLUME = {386},
      YEAR = {2023},
    NUMBER = {1-2},
     PAGES = {375--414},
      ISSN = {0025-5831,1432-1807},
   MRCLASS = {46L52 (42B20 46L53)},
  MRNUMBER = {4585152},
       DOI = {10.1007/s00208-022-02401-z},
       URL = {https://doi.org/10.1007/s00208-022-02401-z},
}

@article {HongXu2021JFA,
    AUTHOR = {Hong, Guixiang and Xu, Bang},
     TITLE = {A noncommutative weak type {$(1,1)$} estimate for a square
              function from ergodic theory},
   JOURNAL = {J. Funct. Anal.},
  FJOURNAL = {Journal of Functional Analysis},
    VOLUME = {280},
      YEAR = {2021},
    NUMBER = {9},
     PAGES = {Paper No. 108959, 29},
      ISSN = {0022-1236,1096-0783},
   MRCLASS = {42B20 (42B25 46L51 46L52 46L53)},
  MRNUMBER = {4220746},
MRREVIEWER = {Chao\ Zhang},
       DOI = {10.1016/j.jfa.2021.108959},
       URL = {https://doi.org/10.1016/j.jfa.2021.108959},
}

@article {CaspersPotapovSukochevZanin2019AmerJMath,
    AUTHOR = {Caspers, Martijn and Potapov, Denis and Sukochev, Fedor A. and Zanin, Dmitriy},
     TITLE = {Weak type commutator and {L}ipschitz estimates: resolution of
              the {N}azarov-{P}eller conjecture},
   JOURNAL = {Amer. J. Math.},
  FJOURNAL = {American Journal of Mathematics},
    VOLUME = {141},
      YEAR = {2019},
    NUMBER = {3},
     PAGES = {593--610},
      ISSN = {0002-9327,1080-6377},
   MRCLASS = {46L10},
  MRNUMBER = {3956516},
       DOI = {10.1353/ajm.2019.0019},
       URL = {https://doi.org/10.1353/ajm.2019.0019},
}

@misc{Lai2026BR,
  author = {Xudong Lai},
  title = {Weak type $(1,1)$ boundedness of Bochner--Riesz means at the critical index on quantum tori},
  year = {2026},
  eprint = {2610.03107},
  archivePrefix = {arXiv},
  primaryClass = {math.OA},
  note = {arXiv:2610.03107v1},
  url = {https://arxiv.org/abs/2610.03107v1}
}

@inproceedings{Parcet2026ICM,
  author = {Parcet, Javier},
  title = {The Impact of {Schur} Multipliers in Harmonic Analysis and Operator Algebras},
  booktitle = {Proceedings of the International Congress of Mathematicians 2026 - Volume 4: Invited Lectures (Sections 5--8)},
  publisher = {Society for Industrial and Applied Mathematics},
  year = {2026},
  pages = {730--750},
  doi = {10.1137/25M1803371},
  url = {https://doi.org/10.1137/25M1803371}
}

\end{document}